\documentclass[11pt,letterpaper]{amsart}
\usepackage{amsmath}
\numberwithin{equation}{section}
\usepackage{amssymb,esint,hyperref}
\usepackage{amscd}
\usepackage{xspace}
\usepackage{fancyhdr}
\usepackage{color}
\usepackage{verbatim}
\usepackage{cite}
\usepackage{stmaryrd}
\usepackage{bbm}
\usepackage{mathrsfs}
\usepackage{graphicx}
\usepackage{srcltx} 

\newtheorem{theorem}{Theorem}[section]

\newtheorem{lemma}[theorem]{Lemma}

\newtheorem{remark}[theorem]{Remark}

\catcode`\@=11\global\let\AddToReset=\@addtoreset
\AddToReset{equation}{section}

\AddToReset{theorem}{section}

\title{The Muskat problem with a large slope}
\author{Yiran Xu}
\thanks{E-mail address: yrxu20@fudan.edu.cn, Fudan University, 220 Handan Road, Yangpu, Shanghai, 200433, China.}
\author{Stephen Cameron}
\thanks{E-mail address: scameron@math.uchicago.edu,Department of Mathematics, University of Chicago, 5734 S. University Ave., Chicago, IL 60637}
\author{Ke Chen}
\thanks{E-mail address: k1chen@polyu.edu.cn, Department of Applied Mathematics, The Hong Kong Polytechnic University, Kowloon, Hong Kong, PR China.}
\author{Ruilin Hu}\thanks {E-mail address: huruilin16@mails.ucas.ac.cn,  Academy of Mathematics and Systems Science, Chinese Academy of Sciences, Beijing, 100190, China.}
\author{Quoc-Hung Nguyen}\thanks{E-mail address: qhnguyen@amss.ac.cn, Academy of Mathematics and Systems Science, Chinese Academy of Sciences, Beijing, 100190, China.}
\begin{document}
	\maketitle
	\begin{abstract}	
 In this paper, we establish local well-posedness results for the Muskat equation in any dimension using modulus of continuity techniques. By introducing a novel quantity  \(\beta_\sigma(f_0')\) which encapsulates local monotonicity and slope, we identify a new class of initial data within  \(W^{1,\infty}(\mathbb{R}^d)\). This includes scenarios where the product of the maximal and minimal slopes is large, thereby guaranteeing the local existence of a classical solution.\\
	\end{abstract}
	\section{Introduction and Main Results}
	
	The Muskat equation (see \cite{darcy1856fontaines}, \cite{Mus34}) models the flow of two incompressible, immiscible fluids within a porous medium, originally formulated to describe the dynamics of oil-water interfaces in petroleum engineering. Beyond oil recovery, the Muskat equation finds applications in environmental engineering and water resource management, where understanding fluid interface evolution is critical. Mathematically, the Muskat problem is represented by a nonlinear, fractional-order degenerate parabolic equation that captures the evolution of the interface along with complex fluid dynamics phenomena, including diffusion, convection, and gravity.

This equation bears similarities to models such as the Hele-Shaw flow and the surface quasi-geostrophic (SQG) equation, both of which describe nonlinear interface evolution and instability, reflecting common features of interface-driven systems in fluid dynamics.
	
The flow of fluids within porous media is generally governed by Darcy's Law, represented by the system	\begin{equation*}
		\left\{
		\begin{aligned}
			&\partial_t\rho+\text{div}(\rho v)=0,\quad\text{	(transport equation)}\\
			& \text{div}(v)=0,\quad \quad \quad \ \quad  \text{(incompressible condition)}\\
			&v+\nabla (P+\rho gy)=0,\quad~\text{	(Darcy's law)}
		\end{aligned}
		\right.
	\end{equation*}
	where $g>0$ is the acceleration of gravity. Darcy's law establishes a relationship between the fluid velocity $v$ and the pressure gradient $P$ across the porous medium. The Muskat problem, rooted in this flow model, investigates the dynamics and evolution of the interface separating two immiscible fluids within such media.

Consider a time-dependent curve $\Sigma(t)$ that separates two domains $\Omega_1(t)$ and $\Omega_2(t)$. If we assume that $\Sigma(t)$ can be represented as the graph of a function, we define the domains and interface as follows: 
\begin{align*}
\Omega_1(t)&=\left\{ (x,y)\in \mathbb{R}^d\times \mathbb{R}\,;\, y>f(t,x)\right\},\\
\Omega_2(t)&=\left\{ (x,y)\in \mathbb{R}^d\times \mathbb{R}\,;\, y<f(t,x)\right\},\\
\Sigma(t)&=\left\{ (x,y)\in \mathbb{R}^d\times \mathbb{R}\,;\, y=f(t,x)\right\}.
\end{align*}

We assume that each domain $\Omega_j$, $j=1,2$,  is occupied by an incompressible fluid with constant density $\rho_j$,  where $\rho=\rho(t,x)$ takes the value $\rho_j$ in $\Omega_j(t)$. Assuming $\rho_2>\rho_1$, the heavier fluid lies beneath the lighter one.

The reduction of the Muskat problem to an evolution equation for the free surface parameterization $f(t,x)$ has been well-studied (see~\cite{CaOrSi-SIAM90,EsSi-ADE97,Siegel2004}). Córdoba and Gancedo \cite{Cordoba2007Contour} provided a streamlined formulation, showing that the problem can be reformulated as an evolution equation for the free surface elevation:
	
	\[
	\partial_t f(t,x) = \frac{\rho_2 - \rho_1}{2^d \pi} \text{P.V.} \int_{\mathbb{R}^d} \frac{\alpha \cdot \nabla_x \Delta_\alpha f(t,x)}{\langle \Delta_\alpha f(t,x) \rangle^{d+1}} \frac{d\alpha}{|\alpha|^d},
	\]
	where $\Delta_\alpha f(x) = \frac{f(x) - f(x - \alpha)}{|\alpha|}$, and the integral is interpreted in the sense of principal values.

	In this paper, we focus on the stable regime where $\rho_2 > \rho_1$. To simplify the equation, we set $\frac{\rho_2 - \rho_1}{2^d \pi} = 1$, leading to:
	\begin{align}\label {E1}	
		\partial_t f(t,x) = \text{P.V.} \int_{\mathbb{R}^d} \frac{\alpha \cdot \nabla_x \Delta_\alpha f(t,x)}{\langle \Delta_\alpha f(t,x) \rangle^{d+1}} \frac{d\alpha}{|\alpha|^d}.
	\end{align}
	This equation exhibits invariance under the scaling transformation $f(t,x) \to f_\lambda(t,x) = \frac{1}{\lambda} f(\lambda t, \lambda x)$. A straightforward computation shows that the function spaces $\dot{H}^{1 + \frac{d}{2}}(\mathbb{R}^d)$ and $\dot{W}^{1,\infty}(\mathbb{R}^d)$ are critical for the well-posedness of the Cauchy problem.

There are numerous significant results concerning the Cauchy problem for the Muskat equation with graph interfaces, including global existence under smallness conditions and blow-up scenarios for sufficiently large initial data. These studies capture various singularity formations, such as overhanging interfaces leading to loss of regularity, as demonstrated by Castro et al.\cite{Fefferman2, Fefferman1}, and stability switches, where interfaces initially turn but eventually return to equilibrium, as shown by Córdoba, Gómez-Serrano, and Zlatos \cite{Cor1, Cor2}. Additionally, splash singularities in the one-phase setting have been documented in~\cite{Fefferman1b}, while Shi~\cite{shi, shib} has explored the analyticity of solutions that develop overturned interfaces. For an extensive background on these topics, we refer readers to the excellent surveys by Gancedo~\cite{Gancedo} and Granero-Belinchón-Lazar~\cite{BelinchLazar}.
Foundational work on local well-posedness includes contributions from Yi~\cite{Yi2003}, Ambrose~\cite{Ambrose2004, Ambrose2007}, Córdoba and Gancedo~\cite{Cordoba2007Contour, Cordoba2009}, and Cheng, Granero-Belinchón, and Shkoller~\cite{21}.
	More recently, local well-posedness in sub-critical spaces has been established by Constantin et al.\cite{25} in Sobolev spaces $W^{2,p}(\mathbb{R})$, Ables-Matioc\cite{Ables-Matioc} in $W^{s,p}$ for $s > 1 + \frac{1}{p}$, and Nguyen-Pausader~\cite{43}, Matioc~\cite{Matioc1, Matioc2}, and Alazard-Lazar~\cite{Alazard2020Paralinearization} in $H^s(\mathbb{R})$ with $s > 3/2$.  In the $L^\infty$ setting, it was proved in \cite{KC1} that the 2D Muskat problem is locally well-posed for any initial data in $C^1$.  For further developments, we refer readers to \cite{Self} for self-similar solutions with small initial data, to \cite{Susanna} for the desingularization of small moving corners and to \cite{TH5, peskin, schauder} for broader studies on non-local parabolic equations with structures similar to the Muskat problem.
		
		Due to the parabolic nature of the Muskat equation, these local results imply global existence under smallness assumptions. 
We mention some global existence results with some medium size data. 		Constantin, C{\'o}rdoba, Gancedo, Rodr{\'\i}guez-Piazza 
		and Strain\cite{Constantin2010, Constantin2013, 25} obtained global well-posedness under the condition that the Lipschitz semi-norm of the initial data is less than 1, and proved global classical solutions for initial data in $H^3(\mathbb{R})$ with Wiener norm $\|f_0\|_{\mathcal{L}^{1,1}}=\||\xi|\hat f_0(\xi)\|_{L^1_\xi(\mathbb{R})}$ below a certain threshold (improved in \cite{Constantin2010} to $\frac{1}{3}$). For the 3D case, global classical solutions were shown under $\|f_0\|_{\mathcal{F}^{1,1}}\leq \frac{1}{5}$, with further extensions in \cite{viscojump} to cases involving fluids with differing viscosities. Using modulus of continuity, the second author~\cite{Cameron2020,Cameron2019} also obtained global existence in both 2D and 3D for data with medium size slope, where the upper bounded of slope is 1 in 2D and $5^{-\frac{1}{2}}$ in 3D.
		
		Global solutions with large slopes were later achieved. Deng, Lei, and Lin~\cite{33} proved global existence of weak solution in 2D under monotonicity assumptions. Córdoba and Lazar~\cite{Cordoba-Lazar-H3/2} proved global existence of strong solution in  $\dot H^{\frac{5}{2}} \cap \dot H^{\frac{3}{2}}$ with small critical $\dot H^{\frac{3}{2}}$ norm, and similar result in 3D with small $\dot H^{2}$ critical norm was done by Gancedo and Lazar~\cite{Gancedo}.
			More recently, in a series of works \cite{ThomasHfirst, Alazard2020, Alazard2020endpoint, TH4}, Alazard and the fifth author  demonstrated well-posedness for the Muskat equation with critical initial data that allow  unbounded slopes. Particularly in \cite{Alazard2020endpoint}, leveraging null-type structures to handle the equation's degeneracies, they proved global well-posedness for small initial data in $\dot{H}^{\frac{3}{2}}$ and local well-posedness for large initial data, which is the optimal result in 2D.

	In the present work, we consider the Muskat problem with a large slope. 
	The main techniques in this paper was introduced  in \cite{Kiselev} to prove the global well-posedness of 2D dissipative quasi-geostrophic equation. The equation reads 
	\begin{align*}
		\theta_t-u\cdot \nabla \theta+(-\Delta )^{\frac{1}{2}}\theta=0.
	\end{align*}
	The main idea is to	construct a special
	family of modulus of continuity that are preserved by the dissipative evolution,
	which will lead to a priori estimate for $\|\nabla\theta\|_{L^\infty}$ independent of
	time. To do this, they assume that $\theta$ has  modulus of continuity $\omega$ for any $t<T$, which means that 
	$$|\theta(t,x)-\theta(t,y)|\leq \omega(|x-y|), \ \ \forall 0\leq t<T,
	$$
	for some increasing continuous concave function $\omega: [0,+\infty)\to [0,+\infty)$ such that $\omega(0)=0$.
	By an elementary discussion, they found that the only scenario in which the modulus of continuity $\omega$
	may be lost by $\theta$ after $T$ is the one in which there exist two points $x\neq y$ such that $\theta(T,x)-\theta(T,y)=\omega(|x-y|)$. Then they designed a special function $\omega$ such that  $$\left.\frac{\partial}{\partial t}(\theta(t,x)-\theta(t,y))\right|_{t=T}<0.$$  This contradicts the assumption that the modulus of
	continuity $\omega$ is preserved up to the time $T$. Hence $\theta$ has the  modulus of
	continuity $\omega$ for all times $t\in[0,+\infty)$.

	Applying the techniques in \cite{Kiselev}, the second author proved global well-posedness results in both 2D \cite{Cameron2019} and 3D \cite{Cameron2020}. In 2D, due to the special structure of the equation, he proved the global wellposedness for initial data $f_0\in W^{1,\infty}(\mathbb{R})$ satisfying 
	\begin{align}\label{conste}
		\left(\sup _{x \in \mathbb{R}} f_{0}^{\prime}(x)\right)\left(\sup _{y \in \mathbb{R}}-f_{0}^{\prime}(y)\right)<1.
	\end{align}
	We remark that this condition can be considered as an interpolation between the slope less than 1 assumption in \cite{Constantin2013} and the monotonicity assumption in \cite{33}. Following the scheme in \cite{Kiselev} and using a rescaling argument,  they proved that 
	\begin{align*}
		\left.\frac{d}{dt}(f_x(t,x)-f_x(t,y))\right|_{t=T}<\left.\frac{d}{dt}\left(\rho\left(\frac{x-y}{t}\right)\right)\right|_{t=T},
	\end{align*}
	which is a key to obtain contradiction.
	Hence 	there exists global classical solution satisfying 
	$$
	f_{x}(t, x)-f_{x}(t, y) \leq \rho\left(\frac{|x-y|}{t}\right).
	$$
	We note the sole reason to require \eqref{conste} is to ensure the ellipticity of the kernel $K$ as we defined in \eqref{defK}. In this paper, we prove a local version of the results in \cite{Cameron2019}. More precisely, we propose a local version (see \eqref{defbet}) of the condition \eqref{conste}. Under this condition, the kernel $K$ may not be elliptic, but the main part $K_1$ (see Section \ref{secmain} for definition) is still elliptic. We note that in our proof, the derivative $\left.\frac{\partial}{\partial t}(f(t,x)-f(t,y))\right|_{t=T}$ may be positive due to some remainder terms associated to the  kernel $K-K_1$. To control these terms, we introduce the modulus of
	continuity in the form $\rho(\cdot/t)e^{\tilde C t}$, where  $\tilde C$ is a fixed positive constant. 
	To obtain contradiction as in \cite{Kiselev} and \cite{Cameron2019}, it suffices to prove that 
	\begin{align}\label{contrad}
		\left.\frac{d}{dt}(f_x(t,x)-f_x(t,y))\right|_{t=T}<\left.\frac{d}{dt}\left(\rho\left(\frac{x-y}{t}\right)\right)\right|_{t=T}+\tilde C \rho\left(\frac{x-y}{T}\right).
	\end{align}
	Note that the term $$\tilde C \rho\left(\frac{x-y}{T}\right)$$ is devoted to absorb remainder terms associated to the  kernel $K-K_1$. We remark that we prove \eqref{contrad} directly  instead of using the rescaling argument.
	
	In 3D, the second author proved the global well-posedness in \cite{Cameron2020} under the condition $\|\nabla f_0\|_{L^\infty}\leq 5^{-\frac{1}{2}}$. In this paper, we also prove local well-posedness for high dimension $d\geq 2$ with initial data satisfying $$	\lim_{\varepsilon\to 0}\|\nabla f_0-\nabla f_0\ast \phi_\varepsilon\|_{L^\infty}\leq c_d,$$ where $\phi_\varepsilon=\varepsilon^{-d}\phi(\cdot/\varepsilon)$ is the standard mollifier. We remark that our results hold for any $c_d<\frac{1}{4(d+1)}$. Specifically, in 3D, the constant can be improved to $5^{-\frac{1}{2}}$, which is consistent with the global result in \cite{Cameron2020}.
	
	We introduce some notations that will be used throughout the paper.  Fix a standard symmetric cutoff function $\chi$ such that 
	\begin{align*}
		\chi(x)=\begin{cases}
			1, \ \ &\text{if}\ |x| \leq 1,\\
			0,\ \ & \text{if}\ |x|>2,
		\end{cases}
	\end{align*}
	with  $ |\chi^{(k)}|\leq 2^k.$ 	
	For $\sigma>0$ and $z\in \mathbb{R}$, define
	\begin{align*}
		\chi_{\sigma, z} = \chi\left(\frac{x-z}{\sigma}\right), 
	\end{align*}
	and 
	\begin{align}\label{defbet}
		\beta_\sigma(g)=\sup_{z\in \mathbb{R} }\left(\sup_{x}(\chi_{\sigma, z} g)(x)\sup_{y}(- \chi_{\sigma, z}g)(y)\right).
	\end{align}
	\begin{theorem}\label{thm}
		Let $f_0\in W^{1,\infty}(\mathbb{R})$ with 
		\begin{align}\label{inibe}
			\beta_\sigma(f_0')\leq 1-\varepsilon_0,
		\end{align}
		for some $\sigma, \varepsilon_0>0$. 
		Then there exists $T_0>0$ and a unique classical solution 
		\begin{align*}
			f\in C([0,T_0]\times \mathbb{R})\cap C^{\infty}_{loc}((0,T_0]\times \mathbb{R})
		\end{align*}
		satisfying $\sup_{t\in[0,T_0]}\beta_\sigma(f_x(t))<1$ and
		$$
		\sup_x|\delta_\alpha f_x(t,x)|\leq C_0t^{-1}|\alpha|\ \ \ t\in(0,T_0],\ \forall\alpha\neq 0,
		$$ 
		where $C_0$ is a constant depending on $\|f_0\|_{L^\infty}, \|\partial_x f_0\|_{L^\infty}$, $\varepsilon_0$ and $\sigma$.
	\end{theorem}
	The following lemma provides an example of initial data that satisfies \eqref{inibe} and demonstrates that the product of the maximum and minimum slopes can be made arbitrarily large.
	\begin{lemma}\label{zz}
		Let $\{a_k\}_{k=-\infty}^{+\infty}\subset \mathbb{R}$ be any sequence satisfying $\inf_k (a_k-a_{k-1})\geq \mu_0$ for some $\mu_0>0$. Denote $I_k=(a_k,a_{k+1})$. Then Theorem \ref{thm} holds if the initial data $f_0\in W^{1,\infty}(\mathbb{R})$ satisfies 
		\begin{align*}
			&	\sup_k (\sup_{x\in I_k}f_0'(x)\sup_{y\in I_k}(-f_0'(y)))\leq 1-\varepsilon_0,\\
			&\liminf_{\varepsilon\to 0}\left(	\sup_k \sup_{|x-a_k|\leq \mu_1}|f_0'-f_0'\ast \phi_\varepsilon|(x)\right)\leq 1-\varepsilon_0,
		\end{align*}
		for some $\varepsilon_0>0$ and $\mu_1\in (0,\mu_0/100)$.
	\end{lemma}
	
	\begin{figure}[h]\label{figure}
		\centering
		\includegraphics[width=6cm,height=4cm]{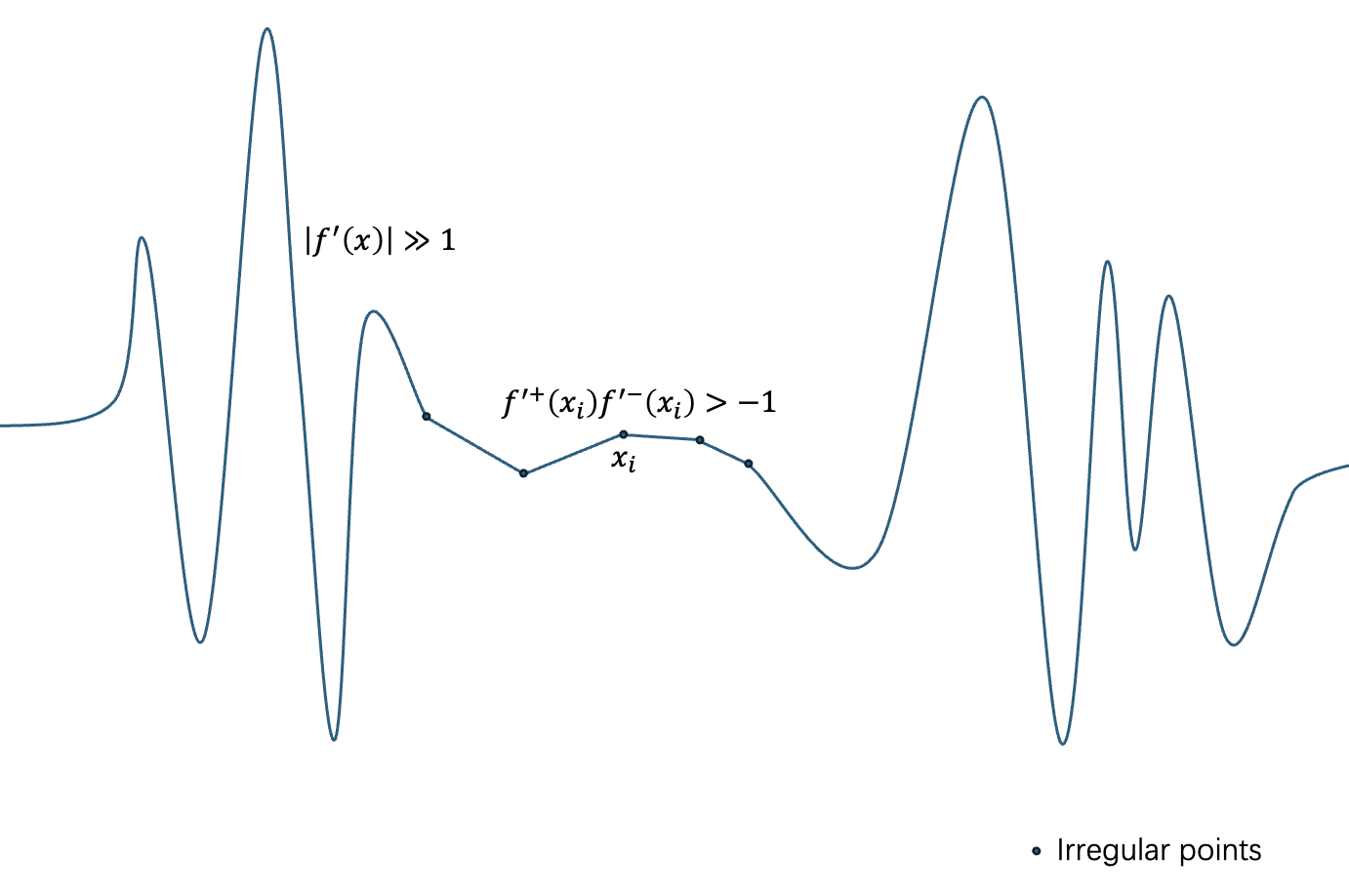}
		\caption{An example of graph}
	\end{figure}
	
	\begin{remark}
		Generally, the condition \eqref{inibe} is not met even for some smooth functions. For example, if we take $f_{0}^\prime(x)=\sin x^2$, then we don't have $\beta_\sigma (f_0^\prime)<1$ for any $\sigma>0$. However, we claim that the condition \eqref{inibe} is true for functions satisfying (see Figure\ref{figure})\\
		(i) piece-wise $C^1$ (with finite irregular points) in a finite interval $K$. \\
		(ii) the derivative is uniformly continuous outside the set $K$.\\
		(iii) $f_0^{\prime +}(x_0)f_0^{\prime -}(x_0)>-1+2\epsilon_0$ for any irregular point $x_0$.\\
		Here $f_0^{\prime +},f_0^{\prime -}$ denote the right and left hand derivatives, respectively. And the uniform continuity means that 
		$$
		\forall \epsilon>0, \exists\delta>0, \text{such that}\ |x-y|<\delta \Rightarrow |f(x)-f(y)|<\epsilon, \ \ \forall x,y\notin K.
		$$		
		
		In fact, if we let $\{x_1,...x_n\}$ be the set of irregular points, then by the condition (i), one  can take  open intervals $O_i=(x_i-\widetilde{\sigma}_i,x_i+\widetilde{\sigma}_i)$ of $x_i$ such that $f_0^{\prime +}(y)f_0^{\prime -}(z)>-1+\epsilon_0$ for any points $y,z\in O_i$. We take smaller open intervals $\widetilde{O}_i=(x_i-\frac{\widetilde{\sigma}_i}{2},x_i+\frac{\widetilde{\sigma}_i}{2})\Subset O_i$ and $f_0^\prime$ is uniformly continuous in $(K-\cup_i \widetilde{O}_i)\cup K^c$. Let $\epsilon=\frac{1-\epsilon_0}{8}$, by (ii), there exists $\delta>0$ such that $$
		|f_0'(x)-f_0'(y)|<\epsilon,\ \ \ \ \forall x,y\notin K, \ |x-y|<\delta.
		$$ 
		Then we obtain $\beta_\sigma(f_0')<1$ by taking $\sigma=\min\{\delta, \frac{\widetilde{\sigma}_i}{3}\}$ to get the result. \end{remark}
	More generally, we have the following theorem for multi-dimensional Muskat equation.
	\begin{theorem}\label{thmhd}
		For any $c_d<\frac{1}{4(d+1)}$ and  any initial data $f_0\in W^{1,\infty}(\mathbb{R}^d)$ such that 
		$$	\lim_{\varepsilon\to 0}\|\nabla f_0-\nabla f_0\ast \phi_\varepsilon\|_{L^\infty}\leq c_d,$$
		there exists $\tilde T_0>0$ and a unique local solution $f\in C([0,\tilde T_0]\times \mathbb{R}^d)\cap C^{\infty}_{loc}((0,\tilde T_0]\times \mathbb{R}^d)$  to the Cauchy problem of \eqref{E1} satisfying 
		\begin{align*}
			\sup_x |\delta_\alpha \nabla f(t,x)|\leq C_0(\|\nabla f_0\|_{L^\infty},c_d)t^{-1}|\alpha|, \ \ \ t\in (0,\tilde T_0].
		\end{align*}
	\end{theorem}
	
	\section{Proof of Theorem \ref{thm}}
	Let $f\in C([0,T];C^3(\mathbb{R}))$ be a solution to the Muskat equation with initial data satisfying $f_0'\in C_c^\infty (\mathbb{R})$. Assume for any $0<t\leq T$,
	\begin{align}
		&\beta_\sigma(t)\leq1-\varepsilon_0,\label{con1}\\	
		&f_x(t,x)-f_x(t,y)\leq  \rho\left(\frac{x-y}{t}\right):=\rho_t(x-y),\ \ \forall x\neq y\in\mathbb{R}, \label{con2}\\
		&\|f_x(t)\|_{L^\infty} \leq 2\|f_0'\|_{L^\infty}, \label{con3}
	\end{align} 
	where the  modulus function $\rho$ is concave and there exists a constant $C_0$ such that \begin{align}\label{conrho}
		\rho(\alpha)\leq C_0\alpha,\ \ \ \forall \alpha \geq 0.
	\end{align}
	
	We want to prove that with a specific choice of the function $\rho$,  the properties \eqref{con1}-\eqref{con3} can be extended for a short time if $T<T_0$ for some $T_0>0$ that will be specified later. 
	The following are the main lemmas:
	\begin{lemma}\label{lipest}
		Suppose \eqref{con1}-\eqref{conrho} hold. For any $M>0$, let $f$ be a solution to the Muskat equation with initial data $f_0$ satisfying $\|f_0'\|_{L^\infty}\leq M$. Then there exists $T_1=T_1(\sigma,C_0,M)>0$ such that 
		\begin{align*}
			\sup_{t\in[0,\min\{T,T_1\}]}\|f_x(t)\|_{L^\infty}<\frac{3\|f_0'\|_{L^\infty}}{2}.
		\end{align*}
	\end{lemma}
	\begin{lemma}\label{beta}
		Suppose \eqref{con1}-\eqref{conrho} hold. For any $M>0$, let $f$ be a solution to the Muskat equation with initial data $f_0$ satisfying $\|f_0'\|_{L^\infty}\leq M$. Then there exists $T_2=T_2(\sigma,C_0,M)>0$ such that 
		\begin{align*}
			\sup_{t\in[0,\min\{T,T_2\}]}\beta_\sigma(t)\leq 1-\frac{3\varepsilon_0}{4}.
		\end{align*}
	\end{lemma}
	\begin{lemma}\label{lemmain}
		There exists a function $\rho$ such that if \eqref{con1}-\eqref{conrho} hold, and 
		\begin{align}\label{coeq}
			f_x(T,x)-f_x(T,y)=\rho\left(\frac{x-y}{T}\right)	
		\end{align}
		for some $x\neq y\in\mathbb{R}$ and $T<\sigma$, then there exists $\tilde C>0$ independent of $T$ such that
		\begin{align}\label{mainest}
			\left.\frac{d}{dt}(f_x(t,x)-f_x(t,y))\right|_{t=T}<\left.\frac{d}{dt}\left(\rho\left(\frac{x-y}{t}\right)\right)\right|_{t=T}+\tilde C \rho\left(\frac{x-y}{T}\right).
		\end{align}
	\end{lemma}
	
	\begin{remark}\label{rem}
		The above lemma is also true if we replace the function $\rho$ by $\lambda \rho$ for any $\lambda \in[1,2]$. More precisely, if 
		\begin{align*}
			f_x(T,x)-f_x(T,y)=\lambda\rho\left(\frac{x-y}{T}\right),
		\end{align*}
		for some $x\neq y\in\mathbb{R}$ and $T<\sigma$,	then
		\begin{align*}
			\left.\frac{d}{dt}(f_x(t,x)-f_x(t,y))\right|_{t=T}<\lambda\left.\frac{d}{dt}\left(\rho\left(\frac{x-y}{t}\right)\right)\right|_{t=T}+\tilde C \lambda\rho\left(\frac{x-y}{T}\right).
		\end{align*}
	\end{remark}
	Now we can prove Theorem \ref{thm} by the standard bootstrap argument.
	\begin{proof}[Proof of Theorem \ref{thm}] 
		Set $T_0=\frac{1}{2}\min\{T_1,T_2,\sigma,\tilde C^{-1}\}$, where $T_1, T_2$ are defined in Lemma \ref{lipest} and Lemma \ref{beta}.
		We approximate $f_0$ by a sequence of smooth function whose derivative has compact support. More precisely, denote 
		\begin{align*}
			f_{0\varepsilon}=f_0(0)+\int_0^x\left((f_0'\chi_{\varepsilon^{-1}})\ast \phi_\varepsilon\right)(s)ds,
		\end{align*}
		where $\phi_\varepsilon$ is the standard mollifier.
		By the definition of $\beta_\sigma$ in \eqref{defbet}, one can check that $\beta_\sigma (f_0')\leq 1-\varepsilon_0$ implies $\beta_{\sigma/2} (f_{0\varepsilon}')\leq 1-\varepsilon_0$ for $\varepsilon\leq 10^{-10}\sigma$. Denote $f^\varepsilon$ the classical solution associate to initial data $f_{0\varepsilon}$. Let 
		\begin{align*}
			T=\sup&\left\{t\in[0,T_0]:\sup_{\tau\in[0,t]} \beta_{\sigma/2}(f_x^\varepsilon(t))\leq1-\frac{\varepsilon_0}{2},\|f^\varepsilon_x(t)\|_{L^\infty} \leq 2\|f_{0\varepsilon}'\|_{L^\infty},\right.\\
			&\quad\quad\left.f^\varepsilon_x(t,x)-f^\varepsilon_x(t,y)\leq \tilde \rho(|x-y|/t)e^{\tilde Ct},  \forall x\neq y\in\mathbb{R}\right\}.
		\end{align*}
		for some modulus of continuity $\tilde \rho$ satisfying $\tilde \rho(\alpha)\leq C_0|\alpha|$.	By classical regularity theory, we have $T>0$. We want to prove that $T=T_0$. If $T<T_0$, by Lemma \ref{lipest} and Lemma \ref{beta}, we have 
		\begin{align*}
			&\sup_{t\in[0,T]}\|f^\varepsilon_x(t)\|_{L^\infty} \leq \frac{3}{2}\|f_{0\varepsilon}'\|_{L^\infty},\\
			&\sup_{t\in[0,T]}\beta_{\sigma/2}(f_x^\varepsilon(t))\leq1-\frac{3\varepsilon_0}{4}.
		\end{align*}
		Moreover, we observe that by continuity, there exists $t_0\in[0,T]$ such that $\sup_{\tau\in[0,t_0]}\|f^\varepsilon(\tau)\|_{C^3}\leq 10 (1+\varepsilon^{-2})\|f_0\|_{W^{1,\infty}}$. By the definition of $T$, one has
		$$
		\sup_{t\in[t_0,T]}\|f^\varepsilon(t)\|_{\dot W^{2,\infty}}\leq C_0t_0^{-1}e^{\tilde C t_0}.
		$$
		Hence there exists $R=R(\varepsilon)>0$ such that 
		$$
		\sup_{t\in[0,T]}\|f^\varepsilon(t)\|_{C^3}\leq R.
		$$
		By Lemma \ref{lem2de}, we have $f^\varepsilon_{x}(x), f^\varepsilon_{xx}(x)\to 0$ as $x\to \infty$.
		Applying Lemma \ref{strict} with $g=f_x^\varepsilon$, if $f_x^\varepsilon$ is going to lose its modulus after time $T$, then there exist  $x\neq y\in\mathbb{R}$ such that 
		\begin{align}\label{coeqe}
			f_x^\varepsilon(T,x)-f_x^\varepsilon(T,y)=\tilde \rho\left(\frac{x-y}{T}\right)e^{\tilde CT}.	
		\end{align}
		Applying Remark \ref{rem} with $\lambda=e^{\tilde CT}$, there exists a function $\tilde \rho$ (independent of $\varepsilon$) such that if \eqref{coeqe} holds, then
		\begin{align*}
			\left.\frac{d}{dt}(f^\varepsilon_x(t,x)-f^\varepsilon_x(t,y))\right|_{t=T}<e^{\tilde CT}&\left.\frac{d}{dt}\left(\tilde\rho\left(\frac{x-y}{t}\right)\right)\right|_{t=T}+\tilde Ce^{\tilde CT}\tilde\rho\left(\frac{x-y}{T}\right)\\
			&=\left.\frac{d}{dt}\left(\tilde\rho\left(\frac{x-y}{t}\right)e^{\tilde Ct}\right)\right|_{t=T},
		\end{align*}
		which contradicts the fact that $f_x^\varepsilon$ has modulus $\tilde\rho(\cdot/t)e^{\tilde Ct}$ for any $t<T$. Hence $f^\varepsilon_x$ continues to preserve the modulus for $t\in[T,T+\mu]$ for some $\mu$ small enough. By standard bootstrap argument, we have $T=T_0$. We obtain Theorem \ref{thm} with $\rho(\alpha)=e^{\tilde CT_0}\tilde \rho(\alpha)$ by taking the limit $\varepsilon\to 0$.
	\end{proof}
 In the proof of Theorem \ref{thm}, we utilize the following lemma, which is an analogue of \cite[Lemma 2.1]{Cameron2019}.
	\begin{lemma}\label{strict}
		Let $g\in C((0,T_0],C^2(\mathbb{R}^d))$ be any function satisfying $g(x), \nabla g(x)\to 0$ uniformly as $x\to \infty$.  Suppose for $0<T<T_0$, 
		$$
		g(T,x)-g(T,y)<\rho(|x-y|/T)e^{CT}, \ \ \forall x\neq y\in\mathbb{R}^d,
		$$
		for some Lipschitz modulus of continuity $\rho$ with $\rho''(0)=-\infty$. Then
		$$
		g(T+\epsilon,x)-g(T+\epsilon,y)<\rho(|x-y|/(T+\epsilon))e^{C(T+\epsilon)}, \ \ \forall x\neq y\in\mathbb{R}^d,
		$$
		for $\epsilon$ sufficiently small and $T+\epsilon<T_0$. 
	\end{lemma}
	\begin{proof}
		By uniform continuity, we have that for any compact subset $K\subset\mathbb{R}^{2d}\backslash\{(x,x):x\in\mathbb{R}^d\}$, there exists $\epsilon>0$ small enough such that the conclusion holds for any $x,y\in K$. Hence we only need to consider $|x-y|\ll1$ or $|x|,|y|\gg 1$.\\
		We first consider $|x-y|\ll 1$. By $\rho''(0)=-\infty$ we get 
		$$
		|\nabla g(T,x)|<\frac{\rho'(0)}{T}e^{CT}.
		$$
		Hence  $\|\nabla g(T+\epsilon,\cdot)\|_{L^\infty}<\frac{\rho'(0)}{T+\epsilon}e^{C(T+\epsilon)}$ for $\epsilon$ small enough. Then we obtain the conclusion for $|x-y|< \delta$ for some $\delta$ small enough. \\
		Let $R_1, R_2>0$ be such that 
		$$
		\rho\left(R_{1} /(T+\epsilon)\right)e^{C(T+\epsilon)}>\sup_{\alpha,x}|\delta_\alpha g(T+\epsilon,x)|,
		$$ 
		and $|x|\geq R_2$ implies 
		$$\left|g(T+\epsilon, x)\right|<\frac{\rho(\delta /(T+\epsilon))e^{C(T+\epsilon)}}{2},$$
		for $\epsilon>0$ sufficiently small. Take $R=R_1+R_2$, then $|x|\geq R$ implies 
		$$
		g(T+\epsilon,x)-g(T+\epsilon,y)<\rho(|x-y|/(T+\epsilon))e^{C(T+\epsilon)}, \ \ \forall  y\neq x.
		$$
		Then we complete the proof by taking $K=\{(x,y)\in\mathbb{R}^{2d}:|x-y|\geq \delta, |x|,|y|\leq R\}$.
	\end{proof}
 \begin{proof}[Proof of Lemma \ref{zz} ]
		We prove that there exists $\sigma>0$ such that $\beta_\sigma (f_0')\leq 1-\varepsilon_0/2$.  We first claim that there exists $\nu\in(0,\mu_1)$ such that \begin{align}\label{dada}
			\sup_{|x-y|\leq \nu }(-f_0'(x)f_0'(y))\leq 1-\frac{\varepsilon_0}{4}.
		\end{align}
		It suffices to consider the case where $x\in I_k,y\in I_{k+1}, |x-y|\leq \nu$ for some $k\in\mathbb{Z}$.
		Then one has 
		\begin{align*}
			|f_0'(x)-f_0'(y)|&\leq 	|f_0'-f_0'\ast\phi_\varepsilon|(x)+|f_0'-f_0'\ast\phi_\varepsilon|(y)+|f_0'\ast\phi_\varepsilon(x)-f_0'\ast\phi_\varepsilon(y)|\\
			&\leq 2-2\varepsilon_0+|x-y|\varepsilon^{-1}\|f_0'\|_{L^\infty}=2-\varepsilon_0,
		\end{align*}
		We can take $\varepsilon$ small enough such that 
		$$
		|f_0'-f_0'\ast\phi_\varepsilon|(x),|f_0'-f_0'\ast\phi_\varepsilon|(y)\leq 1-\frac{\varepsilon_0}{2}.
		$$
		Then we take $\nu$ small enough such that 
		$$
		|f_0'-f_0'\ast\phi_\varepsilon|(x)+|f_0'-f_0'\ast\phi_\varepsilon|(y)\leq 2 |x-y|\varepsilon^{-1}\|f_0'\|_{L^\infty}\leq 4 \nu\varepsilon^{-1}\|f_0'\|_{L^\infty}\leq \frac{\varepsilon_0}{2}.
		$$
		Hence we obtain 
		\begin{align*}
			|f_0'(x)-f_0'(y)|&\leq 2-\frac{\varepsilon_0}{2},
		\end{align*} which yields that 
		$$
		-f_0'(x)f_0'(y)\leq \left(\frac{	|f_0'(x)-f_0'(y)|}{2}\right)^2\leq 1-\frac{\varepsilon_0}{4}.
		$$
		This completes the proof of \eqref{dada}.
		Then we obtain $\beta_\sigma (f_0')\leq 1-\frac{\varepsilon_0}{4}$ with $\sigma =\frac{\nu}{2}$.
	\end{proof}
	\subsection{Establish the main estimates}\label{secmain}
	In this section, we prove Lemma \ref{lipest}-Lemma \ref{lemmain} under the assumptions \eqref{con1}-\eqref{conrho}.
	\begin{proof}[Proof of Lemma \ref{lipest}] 
		We have the equation
		\begin{equation}\label{eqde}
			\begin{split}
				(f_x)_t(x)=-f_{xx}(x)&\int_{\mathbb{R}}\frac{1}{\langle\Delta_\alpha f(x)\rangle^2}\frac{d\alpha}{\alpha}\\
				&+2\int_{\mathbb{R}}\frac{E_\alpha f(x)(f_x(x)\Delta_\alpha f(x)+1)}{\langle\Delta_\alpha f(x)\rangle^4}\frac{d\alpha}{\alpha^2},	
			\end{split}	
		\end{equation}
		where we denote \begin{equation}\label{eqdefnote}
			E_\alpha f(x)=\Delta_\alpha f(x)-f_x(x),\ \ \ \Delta_\alpha f(x)=\frac{f(x)-f(x-\alpha)}{\alpha},\ \ \ \langle A \rangle=\sqrt{A^2+1}.
		\end{equation}
		It is easy to check that 
		$$
		E_\alpha f(x)=-\fint_0^\alpha \delta_hf_x(x)dh=\begin{cases}
			-\frac{1}{\alpha}\int_{0}^\alpha \delta_hf_x(x)dh,&~\text{if}~ \alpha>0,\\
			\frac{1}{\alpha}\int_\alpha^{0} \delta_hf_x(x)dh,&~\text{if} ~\alpha<0.
		\end{cases}
		$$
		Denote $$k(x,\alpha)=\frac{2(f_x(x)\Delta_\alpha f(x)+1)}{\langle\Delta_\alpha f(x)\rangle^4}.$$
		Then
		\begin{equation}\label{Ealp}	\begin{aligned}
				&\int_{\mathbb{R}}E_\alpha f(x)k(x,\alpha)\frac{d\alpha}{\alpha^2}\\
				&=-\int_0^{+\infty}\int_0^\alpha\delta_h f_x(x) k(x,\alpha)\frac{dhd\alpha}{\alpha^3}+\int_{-\infty}^0\int_\alpha^0\delta_h f_x(x) k(x,\alpha)\frac{dhd\alpha}{\alpha^3}\\
				&=-\int_0^{+\infty}\delta_\alpha f_x(x)\int_\alpha^{+\infty}k(x,h)\frac{dhd\alpha}{\alpha^3}+\int_{-\infty}^0\delta_\alpha f_x(x) \int_{-\infty}^\alpha k(x,h)\frac{dhd\alpha}{\alpha^3}\\
				&=-\int_{\mathbb{R}}\delta_\alpha f_x(x)K(x,\alpha){d\alpha},
			\end{aligned}
		\end{equation}
		where we denote \begin{align}\label{defK}
			K(x,\alpha)=K_1(x,\alpha)+K_2(x,\alpha),
		\end{align} with
		\begin{align*}
			&K_1(x,\alpha)=\int_\alpha^\sigma k(x,h)\frac{dh}{h^3}  \mathbf{1}_{0<\alpha\leq \sigma}-\int_{-\sigma}^\alpha k(x,h)\frac{dh}{h^3}   \mathbf{1}_{-\sigma<\alpha< 0},\\
			&K_2(x,\alpha)=	\int_{\max\{\alpha,\sigma\}}^{+\infty}k(x,h)\frac{dh}{h^3} \mathbf{1}_{\alpha>0}-\int_{-\infty}^{\min\{\alpha,-\sigma\}} k(x,h)\frac{dh}{h^3} \mathbf{1}_{\alpha<0}.
		\end{align*}
		By the definition of $k(x,h)$, it is easy to check that for $|h|\leq \sigma$
		\begin{align*}
			\frac{2(1-\beta_\sigma(f_x))}{(1+\|f_x\|_{L^\infty}^2)^2}	\leq k(x,h)\leq 2(1+\|f_x\|_{L^\infty}).
		\end{align*}
		For simplicity, we let $L=2\|f_0'\|_{L^\infty}$.	Then \eqref{con1} and \eqref{con3} lead to
		\begin{align*}
			0\leq	\frac{\varepsilon_0\mathbf{1}_{|\alpha|\leq \sigma}}{2(1+L^2)^2}\left(\frac{1}{\alpha^2}-\frac{1}{\sigma^2}\right)\leq K_1(x,\alpha)\leq \frac{1+L}{\alpha^2},
		\end{align*}
		and that 
		\begin{align*}
			|K_2(x,\alpha)|\leq \frac{2(1+L)}{\max\{|\alpha|,\sigma\}^2}.
		\end{align*}
		Hence \eqref{eqde} yields
		\begin{align*}
			(f_x)_t(x)=-f_{xx}(x)&\int_{\mathbb{R}}\frac{1}{\langle\Delta_\alpha f(x)\rangle^2}\frac{d\alpha}{\alpha}-\int_{\mathbb{R}}\delta_\alpha f_x(x)K(x,\alpha){d\alpha}.
		\end{align*}
		Let $x_t$ be such that $f_x(t,x_t)=\sup_xf_x(t,x)$.
		Then $f_{xx}(x_t)=0$ and $\delta_\alpha f_x(x_t)\geq 0$. Hence by the ellipticity of $K_1$ we obtain
		\begin{align*}
			\frac{d}{dt}f_x(x_t)&\leq -\int_{\mathbb{R}}\delta_\alpha f_x(x_t)K_2(x_t,\alpha){d\alpha}.
		\end{align*}
		It is easy to check that 
		\begin{align}\label{K2Linf}
			\left|\int_{\mathbb{R}}\delta_\alpha f_x(x_t)K_2(x_t,\alpha){d\alpha}\right|\lesssim (1+L)^2\int_{\mathbb{R}} \frac{1}{\max\{|\alpha|,\sigma\}^2}d\alpha\lesssim (1+L)^2\sigma^{-1}.
		\end{align}	
		The parallel argument holds for $\sup_x(-f_x(t,x))$. Then we conclude that 
		\begin{align*}
			\frac{d}{dt}\|f_x(t)\|_{L^\infty}\leq C_1  (1+L)^2\sigma^{-1}.
		\end{align*} 
		Integrate in time we have
		\begin{align*}
			\|f_x(t)\|_{L^\infty}&\leq  \|f_0'\|_{L^\infty}+C_1(1+L)^2\sigma^{-1}t.
		\end{align*}
		Then we get 
		\begin{align*}
			\sup_{t\in[0,T_1]}\|f_x(t)\|_{L^\infty}<\frac{3}{2}\|f_0'\|_{L^\infty},
		\end{align*}
		for any $	T_1<\frac{\sigma \|f_0'\|_{L^\infty}}{10C_1 (1+L)^2}$. This completes the proof.
	\end{proof}

	\begin{proof}[Proof of Lemma \ref{beta}]
		For any $z\in\mathbb{R}$, we write $\chi_\sigma =\chi_{\sigma,z}$ for simplicity. Multiply $\chi_\sigma(x)$ on both sides of \eqref{eqde}, we obtain from \eqref{Ealp} that 
		\begin{align*}
			(\chi_\sigma f_x)_t(x)&=-\chi_\sigma(x) f_{xx}(x)\int_{\mathbb{R}}\frac{1}{\langle\Delta_\alpha f(x)\rangle^2}\frac{d\alpha}{\alpha}-\chi_\sigma(x)\int_{\mathbb{R}}\delta_\alpha f_x(x)K(x,\alpha)d\alpha\\
			&=-((\chi_\sigma f_{x})_x(x)-\chi_\sigma'(x)f_x(x))\int_{\mathbb{R}}\frac{1}{\langle\Delta_\alpha f(x)\rangle^2}\frac{d\alpha}{\alpha}\\
			&\ \ \ \quad -\int_{\mathbb{R}}\delta_\alpha (\chi_\sigma f_x)(x)K(x,\alpha)d\alpha-\int_{\mathbb{R}}\delta_\alpha \chi_\sigma(x)f_x(x-\alpha)K(x,\alpha)d\alpha.
		\end{align*}
		Let $x_t$ such that $(\chi_\sigma f_x)(t,x_t)=\sup_x(\chi_\sigma f_x)(t,x)$.
		Then $(\chi_\sigma f_x )_x(t,x_t)=0$ and $\delta_\beta(\chi_\sigma f_x)(t,x_t)\geq 0$. Combining this with the fact $K(x,\alpha)=K_1(x,\alpha)+K_2(x,\alpha)\geq K_2(x,\alpha)$ to get
		\begin{align*}
			\frac{d}{dt}\left((\chi_\sigma f_x)(x_t)\right)\leq&\chi_\sigma'(x)f_x(x)\int_{\mathbb{R}}\frac{1}{\langle\Delta_\alpha f(x)\rangle^2}\frac{d\alpha}{\alpha}\\
			& -\int_{\mathbb{R}}\delta_\beta(\chi_\sigma f_x)(x)K_2(x,\beta)d\beta\\	&-\int_{\mathbb{R}}\delta_\alpha \chi_\sigma(x)f_x(x-\alpha)K(x,\alpha)d\alpha\\
			=:&I_1+I_2+I_3.
		\end{align*}
		Recall that $\|\chi_\sigma'\|_{L^\infty}\leq 2\sigma^{-1}$. By symmetry we have 
		\begin{align*}
			|I_1|\lesssim  \sigma^{-1}L\int_{\mathbb{R}}\left|\frac{1}{\langle\Delta_\alpha f(x)\rangle^2}-\frac{1}{\langle\Delta_{-\alpha} f(x)\rangle^2}\right|\frac{d\alpha}{|\alpha|}.
		\end{align*}
		It is easy to check that for any $A_1, A_2\in\mathbb{R}$,
		\begin{align*}
			\left|\frac{1}{\langle A_1\rangle^2}-\frac{1}{\langle A_2\rangle^2}\right|\leq |A_1-A_2|.
		\end{align*}
		Moreover, we have \begin{align*}
			|\Delta_{\alpha} f(x)-\Delta_{-\alpha} f(x)|&\lesssim  \min\{\rho_t(\alpha),\|f_x\|_{L^\infty},\|f\|_{L^\infty}\alpha^{-1}\}\\
			&\lesssim \min\{C_0t^{-1}|\alpha|,L,\|f_0\|_{L^\infty}\alpha^{-1}\}.
		\end{align*}
		Hence 
		\begin{equation}\label{I1}
			\begin{aligned}
				|I_1|&\lesssim \sigma^{-1}L\int_{\mathbb{R}}|\Delta_{\alpha} f(x)-\Delta_{-\alpha} f(x)|\frac{d\alpha}{|\alpha|}\\&\lesssim 	\sigma^{-1}L\left(\frac{C_0}{t}\int_{|\alpha|\leq \frac{t}{C_0} }{d\alpha}+L\int_{\frac{t}{C_0}\leq|\alpha|\leq\|f_0\|_{L^\infty}}\frac{d\alpha}{|\alpha|}+\|f_0\|_{L^\infty}\int_{|\alpha|> \|f_0\|_{L^\infty} }\frac{d\alpha}{|\alpha|^2}\right)\\\
				&	\lesssim (1+L)^2\sigma^{-1}(1+\log({C_0t^{-1}\|f_0\|_{L^\infty}})).
			\end{aligned}
		\end{equation}
		Then, following \eqref{K2Linf} to obtain 
		\begin{align}\label{I2}
			|I_2|\lesssim (1+L)^2\sigma^{-1}.
		\end{align}
		Finally, we estimate $I_3$. When $|\alpha| $ is small, we approximate $\delta_\alpha \chi_\sigma(x)$ by $\alpha\chi_\sigma'(x)$, $f_x(x-\alpha)$ by $f(x)$, and $K(x,\alpha)$ by $\frac{1}{\langle f_x(x)\rangle^2\alpha^2}$. More precisely, we can write $I_3$ as 
		\begin{align*}
			I_3=&\int_{\mathbb{R}}\delta_\alpha \chi_\sigma(x)\delta_\alpha f_x(x)K(x,\alpha)d\alpha
			-f_x(x)\int_{|\alpha|\geq \sigma}\delta_\alpha \chi_\sigma(x)K(x,\alpha)d\alpha\\
			&\ \ -f_x(x)\int_{|\alpha|\leq \sigma}(\delta_\alpha \chi_\sigma(x)-\alpha\chi'_\sigma(x))K(x,\alpha)d\alpha \\
			&\ \ -f_x(x)\chi_\sigma'(x)\int_{|\alpha|\leq \sigma} \alpha \left(K(x,\alpha)-\frac{1}{\langle f_x(x)\rangle^2\alpha^2}\right)d\alpha\\
			:=&I_{3,1}+I_{3,2}+I_{3,3}+I_{3,4},
		\end{align*}
		this follows from the fact that $\int_{|\alpha|\leq \sigma}\frac{d\alpha}{\alpha}=0$.
		Using the fact that $|K(x,\alpha)|\lesssim \frac{1+L}{\alpha^2}$, we obtain 
		\begin{align*}
			|I_{3,1}|&\lesssim(1+L)\int_{\mathbb{R}}|\delta_\alpha \chi_\sigma(x)\delta_\alpha f_x(x)|\frac{d\alpha}{\alpha^2}\\
			&\lesssim (1+L)\int_{\mathbb{R}}\min\{L,\sigma^{-1}|\alpha|\rho_t(\alpha)\}\frac{d\alpha}{\alpha^2}\\
			&\lesssim (1+L)C_0^\frac{1}{2}L^\frac{1}{2}\sigma^{-\frac{1}{2}}t^{-\frac{1}{2}}.
		\end{align*}
		Moreover, it is easy to check that 
		\begin{align*}
			&|I_{3,2}|\lesssim (1+L)^2\int_{|\alpha|\geq \sigma}\frac{d\alpha}{\alpha^2} \lesssim (1+L)^2\sigma^{-1},\\
			&	|I_{3,3}|\lesssim (1+L)^2\int_{\mathbb{R}}|\delta_\alpha \chi_\sigma(x)-\alpha\chi_\sigma'(x)|\frac{d\alpha}{\alpha^2}\lesssim (1+L)^2\int_{\mathbb{R}}\min\{1,\sigma^{-2}|\alpha|^2\frac{d\alpha}{\alpha^2}\}\lesssim L{\sigma^{-1}}.
		\end{align*}
		For $I_{3,4}$, observe that 
		\begin{align*}
			\frac{1}{\langle f_x(x)\rangle^2\alpha^2}=\frac{2}{\langle f_x(x)\rangle^2}\left(\int_{\alpha}^\infty\frac{dh}{h^3}\mathbf{1}_{\alpha>0}-\int_{-\infty}^{\alpha}\frac{dh}{h^3}\mathbf{1}_{\alpha<0}\right).
		\end{align*}	
		Hence 
		\begin{align*}
			&K(x,\alpha)-\frac{1}{\langle f_x(x)\rangle^2\alpha^2}\\
			&\ \ \ =\int_{\alpha}^\infty \left(k(x,h)-\frac{2}{\langle f_x(x)\rangle^2}\right)\frac{dh}{h^3}\mathbf{1}_{\alpha>0}-\int_{-\infty}^{\alpha}\left(k(x,h)-\frac{2}{\langle f_x(x)\rangle^2}\right)\frac{dh}{h^3}\mathbf{1}_{\alpha<0}.
		\end{align*}
		It is easy to check that 
		\begin{align*}
			&\left|k(x,h)-\frac{2}{\langle f_x(x)\rangle^2}\right|\\
			&\lesssim  \left|\frac{f_x^3(x)\Delta_h f(x)+f_x^2(x)+f_x(x)\Delta_h f(x)-2(\Delta_h f(x))^2-(\Delta_h f(x))^4}{\langle f_x(x)\rangle^2\langle \Delta_h f(x)\rangle^4}\right|\\
			&\lesssim  \left|E_h f(x)\right|\left|\frac{(f_x^2(x)+f_x(x)\Delta_h f(x)+(\Delta_h f(x))^2)\Delta_h f(x)+\Delta_h f(x)+f_x(x)} {\langle f_x(x)\rangle^2\langle \Delta_h f(x)\rangle^4}\right|\\
			&\lesssim  (1+L)\left|E_h f(x)\right|,
		\end{align*}
		with $E_hf(x)$ defined as \eqref{eqdefnote}. Hence we obtain that 
		\begin{align*}
			\left|K(x,\alpha)-\frac{1}{\langle f_x(x)\rangle^2\alpha^2}\right|&\lesssim(1+L) \int_{|\alpha|}^\infty|E_h f(x)|\frac{dh}{h^3}.
		\end{align*}
		Substitute  this into the formula of $I_{3,4}$, we have 
		\begin{align*}
			|I_{3,4}|\lesssim C_0(1+L)^2\sigma^{-1}\int _{|\alpha|\leq \sigma}\min\{C_0t^{-1},L|\alpha|^{-1}\}d\alpha\lesssim (1+L+C_0)^3 t^{-\frac{1}{2}}\sigma^{-\frac{1}{2}}.
		\end{align*}
		We conclude that 
		\begin{align*}
			|I_3|\lesssim (1+L+C_0)^3(\sigma^{-\frac{1}{2}} t^{-\frac{1}{2}}+\sigma^{-1}).
		\end{align*}
		Combining this with \eqref{I1} and \eqref{I2}, we obtain
		\begin{align*}
			\frac{d}{dt}\left((\chi_\sigma f_x)(x_t)\right)\leq C_2(1+L+C_0)^3\sigma^{-1}(1+\log({C_0t^{-1}\|f_0\|_{L^\infty}})+\sigma^\frac{1}{2}t^{-\frac{1}{2}}).
		\end{align*}
		Integrate in time we obtain 
		\begin{align*}
			\sup_x (\chi_\sigma f_x)(t,x)\leq \sup_x (\chi_\sigma f_0')(x)+C_2(1+L+C_0+\|f_0\|_{L^\infty})^5\sigma^{-1}t^\frac{1}{2}.
		\end{align*}
		Let $\tilde x_t$ be such that  $-(\chi_\sigma f_x)(t,\tilde x_t)=\sup_x(-\chi_\sigma f_x)(t,x)$. Follow the above estimates, we obtain 
		\begin{align*}
			\sup_x (-\chi_\sigma f_x)(t,x)\leq \sup_x (-\chi_\sigma f_0')(x)+C_2(1+L+C_0+\|f_0\|_{L^\infty})^5\sigma^{-1}t^\frac{1}{2}.
		\end{align*}
		Hence we have 
		\begin{align*}
			\beta_\sigma(t)\leq& \beta_\sigma(0)+2C_2(1+L+C_0+\|f_0\|_{L^\infty})^5\sigma^{-1}t^\frac{1}{2}.
		\end{align*}
		We can choose $T_2$ small enough such that 
		\begin{align*}
			\sup_{t\in[0,T_2]}\beta_\sigma(t)\leq \beta_\sigma(0)+\frac{\varepsilon_0}{4}.
		\end{align*}
		This completes the proof.
	\end{proof}
	The remaining part of this section is devoted to prove Lemma \ref{lemmain}. We first introduce the following lemma.
	\begin{lemma}\label{lemmodulus}
		Suppose $T\leq \sigma$, conditions \eqref{con1}-\eqref{coeq} hold for some modulus function $\rho$ and $x\neq y\in\mathbb{R}$, then 
		\begin{equation}\label{final}
			\begin{aligned}
				&\left.\frac{d}{dt}(f_{x}(x)-f_{x}(y))\right|_{t=T}\\
				\leq& 4(1+L)\rho_T'(z)\left(\int_0^{z}\rho_T\left(h\right)\frac{dh}{h}+z\int_{z}^\infty\rho_T\left(h\right)\frac{dh}{h^2}+\rho_T(z)\right)\\
				&-\frac{\varepsilon_0}{(1+L^2)^2} \int_0^z\delta_\alpha\rho_T (z)+\delta_{-\alpha}\rho_T (z)\frac{d\alpha}{\alpha^2}\\
				&+\frac{\varepsilon_0}{(1+L^2)^2}\int_z^\infty\rho_T (\alpha+z)-\rho_T (\alpha)-\rho_T (z)\frac{d\alpha}{\alpha^2}\\
				&+4(1+L)\rho_T(z)\int_{z}^\infty\rho_T(|\alpha|+z)-\rho_T(z)\frac{d\alpha}{|\alpha|^2}+5(1+L)\sigma^{-1}\rho_T(z),
			\end{aligned}
		\end{equation}
		where $z=|x-y|$.
	\end{lemma}
	\begin{proof}
		Without loss of generality, let $x>y$ and  $z=x-y$.
		Then we have 
		\begin{align*}
			\left.\frac{d}{dt}(f_{x}(x)-f_{x}(y))\right|_{t=T}=&-f_{xx}(x)\int_{\mathbb{R}}\frac{1}{\langle\Delta_\alpha f(x)\rangle^2}\frac{d\alpha}{\alpha}+f_{xx}(y)\int_{\mathbb{R}}\frac{1}{\langle\Delta_\alpha f(y)\rangle^2}\frac{d\alpha}{\alpha}\\
			&-\int_{\mathbb{R}}\left(\delta_\alpha f_x(x)K(x,\alpha)-\delta_\alpha f_x(y)K(y,\alpha)\right){d\alpha}\\
			:=&Q_1+Q_2.
		\end{align*}
		Here we omit the time variable $t=T$ for simplicity.
		Note that by \eqref{con2} and \eqref{coeq}, we have 
		\begin{equation}\label{del}
			\begin{aligned}
				&\delta_\alpha f_x(x)=f_x(x)-f_x(y)+f_x(y)-f_x(x-\alpha)\geq\delta_\alpha \rho_T(z),\ \ \forall \alpha\leq z,\\
				&\delta_\alpha f_x(y)=f_x(y)-f_x(x)+f_x(x)-f_x(y-\alpha)\leq -\delta_{-\alpha}\rho_T(z), \ \ \forall \alpha \geq-z.
			\end{aligned}
		\end{equation}
		Hence we obtain \begin{align}\label{2de}
			f_{xx}(x)=f_{xx}(y)=\rho'_T\left({z}\right).
		\end{align} Then
		\begin{align*}
			Q_1=\rho'_T\left({z}\right)\int_{\mathbb{R}} \frac{1}{\langle\Delta_\alpha f(y)\rangle^2}-\frac{1}{\langle\Delta_\alpha f(x)\rangle^2}\frac{d\alpha}{\alpha}.
		\end{align*}By symmetry we obtain
		\begin{equation}\label{Q1f}
			\begin{aligned}
				&\int_{\mathbb{R}} \frac{1}{\langle\Delta_\alpha f(y)\rangle^2}-\frac{1}{\langle\Delta_\alpha f(x)\rangle^2}\frac{d\alpha}{\alpha}\\
				&\leq \frac{1}{2}\int_{|\alpha|\leq |z|}\left|\frac{1}{\langle\Delta_\alpha f(y)\rangle^2}-\frac{1}{\langle\Delta_{-\alpha} f(y)\rangle^2}\right| +\left|\frac{1}{\langle\Delta_\alpha f(x)\rangle^2}-\frac{1}{\langle\Delta_{-\alpha} f(x)\rangle^2}\right|\frac{d\alpha}{|\alpha|}\\
				&\quad\quad\quad+\int_{|\alpha|> |z|}\left| \frac{1}{\langle\Delta_\alpha f(y)\rangle^2}-\frac{1}{\langle\Delta_\alpha f(x)\rangle^2}\right|\frac{d\alpha}{|\alpha|}.
			\end{aligned} 
		\end{equation}
		Moreover, by \eqref{con2} we have
		\begin{align}
			&|\Delta_\alpha f(x)-\Delta_{-\alpha}f(x)|=\left|\fint_{-\alpha}^0f'(x+h)-f'(x+h+\alpha)dh\right|\leq \rho_T\left({\alpha}\right),\label{alph}\\
			&|\Delta_\alpha f(x)-\Delta_{\alpha}f(y )|=\frac{1}{|\alpha|}\left|\int_{y}^xf'(h)-f'(h-\alpha)dh\right|\leq \frac{|z|}{|\alpha|}\rho_T\left({\alpha}\right).\label{xy}
		\end{align}
		Specifically, by concavity of $\rho$ we have $\frac{\rho_T(\alpha)}{|\alpha|}\leq \frac{\rho_T(z)}{|z|}$ for $|\alpha|>|z|$. Hence 
		\begin{align}\label{xysma}
			|\Delta_\alpha f(x)-\Delta_{\alpha}f(y )|\leq \rho_T(z).
		\end{align}
		Substitute \eqref{alph} and \eqref{xy} in \eqref{Q1f} we obtain
		\begin{align*}
			&\int_{\mathbb{R}} \frac{1}{\langle\Delta_\alpha f(y)\rangle^2}-\frac{1}{\langle\Delta_\alpha f(x)\rangle^2}\frac{d\alpha}{\alpha}\leq \int_{|\alpha|\leq z}\rho_T\left({\alpha}\right)\frac{d\alpha}{|\alpha|}+z\int_{|\alpha|> z}\rho_T\left({\alpha}\right)\frac{d\alpha}{|\alpha|^2}.
		\end{align*}
		Then 
		\begin{align}\label{Q1}
			Q_1\leq \rho'_T(z)\left(\int_{|\alpha|\leq z}\rho_T\left({\alpha}\right)\frac{d\alpha}{|\alpha|}+z\int_{|\alpha|> z}\rho_T\left({\alpha}\right)\frac{d\alpha}{|\alpha|^2}\right).
		\end{align}
		Then we estimate
		\begin{align*}
			Q_2=	\int_{\mathbb{R}}\delta_\alpha f_x(y)K(y,\alpha)-\delta_\alpha f_x(x)K(x,\alpha){d\alpha}.
		\end{align*}
		Recall that $f_{xx}(x)=f_{xx}(y)=\rho_T'\left({z}\right)$. We approximate $\delta_\alpha f_x(x), \delta_\alpha f_x(y)$ by ${\alpha}\rho_T'\left({z}\right)$ when $\alpha$ small, which leads to 
		\begin{equation}\label{3terms}
			\begin{aligned}
				Q_2
				&=\rho_T'\left({z}\right)\int_{|\alpha|\leq {z}}\alpha\left(K_1(y,\alpha)-K_1(x,\alpha) \right)d\alpha\\
				&\quad\quad+	\int_{|\alpha|\leq {z}}\left((\delta_\alpha f_x(y)-\alpha f_{xx}(y)) K_1(y,\alpha)-(\delta_\alpha f_x(x)-\alpha f_{xx}(x)) K_1(x,\alpha)\right)d\alpha\\
				&\quad\quad+\int_{|\alpha|> {z}}\left(\delta_\alpha f_x(y) K_1(y,\alpha)-\delta_\alpha f_x(x) K_1(x,\alpha)\right)d\alpha\\
				&	\quad\quad+\int_{\mathbb{R}}\delta_\alpha f_x(y)K_2(y,\alpha)-\delta_\alpha f_x(x)K_2(x,\alpha){d\alpha}\\
				&=Q_{2,1}+Q_{2,2}+Q_{2,3}+Q_{2,4}.
			\end{aligned}
		\end{equation}
		We first estimate $Q_{2,1}$. We have
		\begin{align*}
			&\int_{|\alpha|\leq {z}}\alpha\left(K_1(y,\alpha)- K_1(x,\alpha)\right)d\alpha\\
			&=\int_0^{z}\alpha\mathbf{1}_{\{\alpha<\sigma\}}\int_\alpha^\sigma k(y,h)-k(x,h)\frac{dh d\alpha}{h^3}-\int_{-{z}}^0\alpha \mathbf{1}_{\{\alpha>-\sigma\}}\int_{-\sigma}^\alpha k(y,h)-k(x,h)\frac{dh d\alpha}{h^3}\\
			&=\int_0^\sigma\int_0^{\min\{h,{z}\}}\alpha d\alpha (k(y,h)-k(x,h)- k(y,-h)+k(x,-h))\frac{dh }{h^3}.
		\end{align*}
		Note that \begin{align*}
			|k(x,h)-k(x,-h)|&\leq (1+\|f_x\|_{L^\infty})|\Delta_h f(x)-\Delta_{-h} f(x)|\nonumber\\&\overset{\eqref{alph}}\leq(1+\|f_x\|_{L^\infty})\rho_T\left(h\right).\nonumber
		\end{align*}
		For $|\alpha|\leq z$, we have 	
		\begin{align}
			|k(x,h)-k(y,h)|&\leq (1+\|f_x\|_{L^\infty})|\Delta_h f(x)-\Delta_h f(y)|+\rho_T(z)\nonumber\\
			&\overset{\eqref{xysma}}\leq 2(1+\|f_x\|_{L^\infty})\rho_T(z).\label{ker}
		\end{align}
		Hence 
		\begin{align*}
			&\int_{|\alpha|\leq {z}}\alpha\left(K_1(x,\alpha)- K_1(y,\alpha)\right)d\alpha\leq 2(1+\|f_x\|_{L^\infty})\left(\int_0^{z}\rho_T\left(h\right)\frac{dh}{h}+\rho_T(z)\right).
		\end{align*}
		We obtain 
		\begin{align}\label{P1}
			Q_{2,1}\leq 2\rho_T'\left({z}\right) (1+\|f_x\|_{L^\infty})\left(\int_0^{z}\rho_T\left(h\right)\frac{dh}{h}+\rho_T(z)\right).
		\end{align}
		Then we estimate $Q_{2,2}$. 
		Denote 
		\begin{align*}
			&G(\alpha)=(\delta_\alpha f_x(y)-\alpha f_{xx}(y)) K_1(y,\alpha)-(\delta_\alpha f_x(x)-\alpha f_{xx}(x)) K_1(x,\alpha).
		\end{align*}
		We can rewrite $G$ in two different ways, namely
		\begin{align*}
			G(\alpha)&=(\delta_\alpha f_x(y)-\delta_\alpha f_x(x)) K_1(y,\alpha)+(\delta_\alpha f_x(x)-\alpha f_{xx}(x))(K_1(y,\alpha)-K_1(x,\alpha)),\\
			&=(\delta_\alpha f_x(y)-\delta_\alpha f_x(x)) K_1(x,\alpha)+(\delta_\alpha f_x(y)-\alpha f_{xx}(y))(K_1(y,\alpha)-K_1(x,\alpha)),
		\end{align*}
		where we  used the fact that $f_{xx}(x)=f_{xx}(y)$. 
		Note that $\delta_\alpha f_x(y)-\delta_\alpha f_x(x)\leq 0$. Hence 
		\begin{align*}
			&(\delta_\alpha f_x(y)-\delta_\alpha f_x(x)) K_1(y,\alpha),(\delta_\alpha f_x(y)-\delta_\alpha f_x(x)) K_1(x,\alpha)\\
			&\leq \frac{\varepsilon_0\mathbf{1}_{|\alpha|\leq \sigma}}{2(1+L^2)^2}\left(\frac{1}{\alpha^2}-\frac{1}{\sigma^2}\right)(\delta_\alpha f_x(y)-\delta_\alpha f_x(x)).
		\end{align*}
		Moreover, recalling \eqref{del} we have,
		\begin{align*}
			&\delta_\alpha f_x(x)-\alpha f_{xx}(x)\geq \delta_\alpha \rho_T(z)-\alpha\rho'_T(z)\geq 0,\ \ \forall \alpha\leq z,\\
			&\delta_\alpha f_x(y)-\alpha f_{xx}(y)\leq -\delta_{-\alpha} \rho_T(z)-\alpha\rho'_T(z)\leq 0,\ \ \forall \alpha\geq -z.
		\end{align*}
		Hence 
		\begin{align*}
			G(\alpha)\leq  \frac{\varepsilon_0\mathbf{1}_{|\alpha|\leq \sigma}}{2(1+L^2)^2}\left(\frac{1}{\alpha^2}-\frac{1}{\sigma^2}\right)(\delta_\alpha f_x(y)-\delta_\alpha f_x(x)),
		\end{align*}
		and 
		\begin{equation}\label{Q22}
			\begin{aligned}
				Q_{2,2}\leq 
				&\frac{\varepsilon_0}{2(1+L^2)^2}\int_{|\alpha|\leq z}\delta_\alpha f_x(y)-\delta_\alpha f_x(x)\frac{d\alpha}{\alpha^2}\\
				&-\frac{\varepsilon_0}{2\sigma^2(1+L^2)^2}\int_{|\alpha|\leq\min\{\sigma,z\}}\delta_\alpha f_x(y)-\delta_\alpha f_x(x) d\alpha\\
				\leq &\frac{\varepsilon_0}{2(1+L^2)^2}\int_{|\alpha|\leq z}\delta_\alpha f_x(y)-\delta_\alpha f_x(x)\frac{d\alpha}{\alpha^2}+\frac{\varepsilon_0}{2\sigma(1+L^2)^2}\rho_T(z).
			\end{aligned}
		\end{equation}
		Next we estimate $Q_{2,3}$. Denote
		\begin{align*}
			\tilde G(\alpha)=\delta_\alpha f_x(y) K_1(y,\alpha)-\delta_\alpha f_x(x) K_1(x,\alpha).
		\end{align*} 
		We also write $\tilde G$ in two different ways 
		\begin{align*}
			\tilde G(\alpha)&=(\delta_\alpha f_x(y)-\delta_\alpha f_x(x)) K_1(y,\alpha)+\delta_\alpha f_x(x)(  K_1(y,\alpha)-K_1(x,\alpha))\\
			&=(\delta_\alpha f_x(y)-\delta_\alpha f_x(x)) K_1(x,\alpha)+\delta_\alpha f_x(y)(  K_1(y,\alpha)-K_1(x,\alpha)).
		\end{align*}
		Observe that 
		$$
		\delta_\alpha f_x(x)=f_x(x)-f_x(y)+f_x(y)-f_x(x-\alpha)\geq \rho_t(z)-\rho_t(|\alpha|+z),
		$$
		$$\delta_\alpha f_x(y)=f_x(y)-f_x(x)+f_x(x)-f_x(y-\alpha)\leq \rho_t(|\alpha|+z)-\rho_t(z).$$
		By \eqref{ker} we have 
		\begin{align*}
			|K_1(y,\alpha)-K_1(x,\alpha)|\leq 2(L+1)\rho_T(z),\ \ \ \forall |\alpha|>z.
		\end{align*}	Hence 	\begin{align*}
			|\tilde G(\alpha)|\leq& \frac{\varepsilon_0\mathbf{1}_{|\alpha|\leq \sigma}}{2(1+L^2)^2}\left(\frac{1}{\alpha^2}-\frac{1}{\sigma^2}\right)(\delta_\alpha f_x(y)-\delta_\alpha f_x(x))\\
			&\quad+|\rho_T(|\alpha|+z)-\rho_T(z)||K_1(y,\alpha)-K_1(x,\alpha)|.
		\end{align*}
		Hence we obtain 
		\begin{align*}
			Q_{2,3}\leq& \frac{\varepsilon_0}{2(1+L^2)^2}\int_{z\leq|\alpha|\leq \sigma} \left(\delta_\alpha f_x(y)-\delta_\alpha f_x(x)\right)\frac{d\alpha}{\alpha^2}+\frac{\varepsilon_0}{2\sigma(1+L^2)^2}\rho_T(z)\\
			&+2(L+1)\rho_T(z)\int_{|\alpha|>z}|\rho_T(|\alpha|+z)-\rho_T(z)|\frac{d\alpha}{|\alpha|^2}.\nonumber
		\end{align*}
		Note that 
		\begin{align*}
			\int_{|\alpha|\geq \sigma} \left(\delta_\alpha f_x(y)-\delta_\alpha f_x(x)\right)\frac{d\alpha}{\alpha^2}\leq 4\sigma^{-1}\rho_T(z).
		\end{align*}
		Then 
		\begin{align}\label{P2}
			Q_{2,3}\leq& \frac{\varepsilon_0}{2(1+L^2)^2}\int_{z\leq|\alpha|} \left(\delta_\alpha f_x(y)-\delta_\alpha f_x(x)\right)\frac{d\alpha}{\alpha^2}+\frac{4\varepsilon_0}{\sigma(1+L^2)^2}\rho_T(z)\\
			&+2(L+1)\rho_T(z)\int_{|\alpha|>z}|\rho_T(|\alpha|+z)-\rho_T(z)|\frac{d\alpha}{|\alpha|^2}.\nonumber
		\end{align}
		Finally we estimate $Q_{2,4}$. By the definition of $K_2$, we have $|K_2(\cdot,\alpha)|\leq (1+L)\max\{|\alpha|,\sigma\}^{-2}$. Hence
		\begin{align*}
			&\int_{|\alpha|\leq {z}}\left(\delta_\alpha f_x(y)K_2(y,\alpha)-\delta_\alpha f_x(x) K_2(x,\alpha)\right)d\alpha\\
			&\quad\quad\leq (1+L)\int_{|\alpha|\leq {z}} \frac{\rho_T(\alpha)}{\max\{|\alpha|,\sigma\}^2}d\alpha
			\leq 2(1+L)\sigma^{-1}\rho_T(z).
		\end{align*}
		On the other hand, by \eqref{ker} we have
		\begin{align*}
			|K_2(y,\alpha)- K_2(x,\alpha)|\leq (1+L)\sigma^{-1}\rho_T(z),\ \ \ \forall |\alpha|>z.
		\end{align*}
		Hence
		\begin{align*}
			&\int_{|\alpha|> {z}}\left(\delta_\alpha f_x(y)-\delta_\alpha f_x(x)\right) K_2(y,\alpha)d\alpha+\int_{|\alpha|> {z}}\delta_\alpha f_x(x)( K_2(y,\alpha)- K_2(x,\alpha))d\alpha\\
			&\leq (1+L)\rho_T(z)	\int_{|\alpha|> {z}}\frac{1}{\max\{|\alpha|,\sigma\}^2}d\alpha+(1+L)\sigma^{-1}\rho_T(z)\\
			&\leq 2(1+L)\sigma^{-1}\rho_T(z).
		\end{align*}
		Hence we obtain 
		\begin{align*}
			Q_{2,4}\leq 4(1+L)\sigma^{-1}\rho_T(z).
		\end{align*}
		Combining this with \eqref{Q1}, \eqref{3terms}, \eqref{P1}-\eqref{P2}, and Lemma \ref{lemdissip}, we obtain \eqref{final}. This completes the proof.
	\end{proof}

	\begin{proof}[Proof of Lemma \ref{lemmain}]
		Denote $\tilde z=\frac{C_0z}{T}$. Suppose  $\rho_T(z)=\omega(\tilde z)$ for some concave function satisfying $\omega(\alpha)\leq|\alpha|$, $\forall \alpha\in\mathbb{R}$. Then
		$$\rho_T'(z)=\frac{C_0}{T}\omega'(\tilde z),\ \ \ \ \rho_T''(z)=\frac{C_0^2}{T^2}\omega''(\tilde z).$$
		{\bf Case 1: $\tilde z\leq \delta$}\\
		Note that 
		$$\int_0^z\frac{\rho_T (\alpha)}{\alpha}d\alpha=\int_0^{\tilde z}\frac{\omega(\alpha)}{\alpha}d\alpha\leq\tilde z,$$
		and 
		\begin{align*}
			\int_z^\infty \frac{\rho_T (\alpha)}{\alpha^2}d\alpha&=\frac{C_0}{T}\int_{\tilde z}^{\infty}\frac{\omega (\alpha)}{\alpha^2}d\alpha\leq \frac{C_0}{T}\int_{\tilde z}^{\delta }\frac{1}{\alpha}d\alpha+\frac{C_0}{T}\int_{\tilde z}^{\delta }\frac{\omega (\alpha)}{\alpha^2}d\alpha\\
			&\leq \frac{C_0}{T}\left(\ln(\delta/\tilde z)+\frac{\omega (\delta)}{\delta }+\int_\delta^\infty\frac{\omega'(\eta)}{\eta}d\eta\right).
		\end{align*}
		Moreover, by concavity we have 
		\begin{align*}
			\int_{|\alpha|>z}\frac{\rho_T(|\alpha|+z)-\rho_T(z)}{|\alpha|^2}d\alpha\leq 2	\int_z^\infty \frac{\rho_T (\alpha)}{\alpha^2}d\alpha.
		\end{align*}
		Then we estimate the negative part. By concavity and monotonicity, we have $\rho_t ( z+\alpha)\leq \rho_t ( z)+\alpha\rho_t '( z)$ and $\rho_t  ( z-\alpha)\leq \rho_t ( z)-\alpha \rho_t '( z)+\frac{\alpha^2}{2}\rho_t ''( z)$. Hence 
		\begin{align*}
			-\int_{0}^{ z} \frac{\delta_{\alpha} \rho_T ( z)+\delta_{-\alpha} \rho_T ( z)}{\alpha^{2}} d \alpha\leq \frac{ z}{2}\rho_T ''( z)=\frac{C_0\tilde z}{2T}\omega''(\tilde z).
		\end{align*}
		Combining this with \eqref{final}, 
		we obtain
		\begin{align*}
			&\left.\frac{d}{dt}(f_{x}(x)-f_{x}(y))\right|_{t=T}\\
			&\leq\frac{4C_0(1+L)}{T}\omega'(\tilde z)\left(3\tilde z+\tilde z\left(\ln(\delta/\tilde z)+\int_\delta^\infty\frac{\omega'(\eta)}{\eta}d\eta\right)\right)+\frac{ C_0\varepsilon_0\tilde z}{4T(1+L^2)^2}\omega''(\tilde z)\\
			&\quad+\frac{4C_0(1+L)}{T}\omega(\tilde z)\left(\ln(\delta/\tilde z)+\frac{\omega (\delta)}{\delta }+\int_\delta^\infty\frac{\omega'(\eta)}{\eta}d\eta\right)+5(1+L)\sigma^{-1}\omega(\tilde z)\\
			&\leq 	\frac{8C_0(1+L)}{T}\tilde z\left(3+\ln(\delta/\tilde z)+\int_\delta^\infty\frac{\omega'(\eta)}{\eta}d\eta\right)+\frac{ C_0\varepsilon_0\tilde z}{4T(1+L^2)^2}\omega''(\tilde z)+5(1+L)\sigma^{-1}\omega(\tilde z).
		\end{align*}
		{\bf Case 2: $\tilde z> \delta$}\\
		We have 
		$$\int_{0}^{ z} \frac{\rho_t (\alpha)}{\alpha} d \alpha \leq \delta+\omega ( \tilde z) \ln \frac{ \tilde z}{\delta}.$$
		Moreover,
		\begin{align*}
			\int_{|\alpha|>z}(\rho_T(|\alpha|+z)-\rho_T(z))\frac{d\alpha}{|\alpha|^2}\leq \frac{2C_0}{T}\left(\frac{\omega(2\tilde z)-\omega(\tilde z)}{\tilde z}+\int_z^\infty\frac{\rho_T'(\alpha+z)}{\alpha}d\alpha\right).
		\end{align*}
		For the negative part, we have 
		\begin{align*}
			\int_{ z}^{\infty} \frac{\rho_T (\alpha+ z)-\rho_T (\alpha)-\rho_T ( z)}{\alpha^{2}} d \alpha&\leq \-\frac{\rho_T (2 z)-2\rho_T ( z)}{ z}+\int_ z^\infty\frac{\rho_T '(h+ z)-\rho_T '(h)}{h}dh\\
			&\leq -\frac{1}{2}\frac{\rho_T ( z)}{ z}=-\frac{1}{2}\frac{C_0}{T}\frac{\omega(\tilde z)}{\tilde z}.
		\end{align*}
		Then \eqref{final} leads to 
		\begin{equation*}
			\begin{aligned}
				&\left.\frac{d}{dt}(f_{x}(x)-f_{x}(y))\right|_{t=T}\\
				&\leq \frac{4C_0(1+L)}{T}\omega'(\tilde z)\left(\delta+\omega ( \tilde z)( \ln \frac{ \tilde z}{\delta} +2)+\tilde z \int_{\tilde z}^\infty\frac{\omega'(\eta)}{\eta}d\eta
				\right)\\
				&\quad\ -\frac{C_0\varepsilon_0}{4(1+L^2)^2T}\frac{\omega(\tilde z)}{\tilde z}+\frac{4C_0(1+L)}{T}\left(\frac{\omega(2\tilde z)-\omega(\tilde z)}{\tilde z}+\int_z^\infty\frac{\rho_T'(\alpha+z)}{\alpha}d\alpha\right)\\
				&+5(1+L)\sigma^{-1}\omega(\tilde z).
			\end{aligned}
		\end{equation*}
		~~\vspace{0.3cm}\\
		The RHS of \eqref{mainest} reads
		\begin{align*}
			\left.\left(\frac{d}{dt}\left(\rho_t(z)\right)+\tilde C \rho_t(z)\right)\right|_{t=T}=-\frac{C_0z}{T^2}\omega'(\tilde z)+\tilde C \omega(\tilde z)=-\frac{\tilde z}{T}\omega'(\tilde z)+\tilde C \omega(\tilde z).
		\end{align*}
		To prove Lemma \ref{lemmain}, it suffices to take $\tilde C= 10(1+L)\sigma^{-1}$, and find $\omega$ such that 
		\begin{align}\label{smal}
			&{10(1+L)\tilde z}\left(3+C_0^{-1}+\ln(\delta/\tilde z)+\int_\delta^\infty\frac{\omega'(\eta)}{\eta}d\eta\right)+\frac{\varepsilon_0\tilde z}{4(1+L^2)^2}\omega''(\tilde z)<0,\quad\quad \forall\ \tilde z\leq \delta,
		\end{align}
		and 
		\begin{equation}\label{lar}
			\begin{aligned}
				&4(1+L)\omega'(\tilde z)\left(\delta+\omega ( \tilde z)( \ln \frac{ \tilde z}{\delta} +2)+\tilde z \left(C_0^{-1}+\int_{\tilde z}^\infty\frac{\omega'(\eta)}{\eta}d\eta\right)\right)\\
				&\quad\quad-\frac{\varepsilon_0}{4(1+L^2)^2}\frac{\omega(\tilde z)}{\tilde z}+4(1+L)\left(\frac{\omega(2\tilde z)-\omega(\tilde z)}{\tilde z}+\int_z^\infty\frac{\rho_T'(\alpha+z)}{\alpha}d\alpha\right)<0,\quad\quad\forall\tilde z>\delta.
			\end{aligned}
		\end{equation}
		To ensure \begin{align}\label{e1}
			10(1+L)\tilde z\left(3+C_0^{-1}+\ln \frac{\delta}{\tilde z}\right)+\frac{\varepsilon_0\tilde z}{4(1+L^2)^2}\omega''(\tilde z)< 0,\ \ \ \forall \tilde z\leq \delta,
		\end{align} 
		we let 
		\begin{align*}
			\omega''(\tilde z)=- \tilde z^{-\frac{1}{2}}.
		\end{align*}
		And take $\delta$ small enough such that 
		\begin{align*}
			\tilde z^{\frac{1}{2}}\left(4+\ln \frac{\delta}{\tilde z}\right)\leq \frac{\varepsilon_0}{100(1+L)^5},\ \ \ \ \forall \tilde z\leq \delta,
		\end{align*}
		which leads to \eqref{e1}.
		This motivates us to define 
		\begin{align*}
			\omega(\alpha)=\alpha-\frac{4}{3}\alpha^\frac{3}{2},\ \ \ \alpha\leq \delta.
		\end{align*}
		On the other hand, the terms $ 4(1+L )\omega'(\tilde z)(\omega ( \tilde z) \ln \frac{ \tilde z}{\delta}+\tilde z\int_ {\tilde z}^\infty\frac{\omega '(\eta)}{\eta}d\eta) $ in \eqref{lar} suggest us to design 
		\begin{align*}
			\omega'( \tilde z)\leq \frac{\gamma}{ \tilde z(\ln (  \tilde z/\delta)+10)}\ \ \ \text{for}\   \tilde z\geq\delta,
		\end{align*}
		where $0<\gamma\ll 1$ will be specified later. 
		Then we obtain 
		\begin{align*}
			\int_\delta^\infty\frac{\omega'(\eta)}{\eta}d\eta\leq \gamma\delta^{-1}\leq \frac{\varepsilon_0  \delta^{-1/2}}{100(1+L)^5}\leq  -\frac{\varepsilon_0  }{100(1+L)^5}\omega''(\tilde z), \ \ \forall \tilde z\leq \delta,
		\end{align*}
		by taking $\gamma$ small enough. Combining this with \eqref{e1} we obtain \eqref{smal}. 
		
		Next we consider \eqref{lar}. 
		We fix $C_0\geq\frac{\omega^{-1}(2L)}{2L}$. By $\omega(\tilde z)\leq 2L$, we have $\tilde z\leq \omega^{-1}(2L)$. By concavity we have 
		\begin{align*}
			\frac{\tilde z}{	\omega(\tilde z)}\leq \frac{\omega^{-1}(2L)}{2L}\leq C_0,
		\end{align*}
		which leads to $\frac{\tilde z}{C_0}\leq \omega(\tilde z)$. Then
		\begin{equation*}
			\begin{aligned}
				4(1+L)\omega'(\tilde z)&\left(\delta+\omega ( \tilde z)( \ln \frac{ \tilde z}{\delta} +2)+\tilde z \left(C_0^{-1}+\int_{\tilde z}^\infty\frac{\omega'(\eta)}{\eta}d\eta\right)\right)\\
				&\leq4(1+L)\omega'(\tilde z)\omega(\tilde z)( \ln \frac{ \tilde z}{\delta} +4)+4(1+L)\omega'(\tilde z)(\delta+\gamma)\\
				&\leq\frac{ 4(1+L)\gamma( \ln \frac{ \tilde z}{\delta}+4)}{ \ln (  \tilde z/\delta)+10}\frac{\omega(\tilde z)}{\tilde z}+\frac{4(1+L)\gamma}{ \tilde z(\ln (  \tilde z/\delta)+10)}.
			\end{aligned}
		\end{equation*}
		We take  $\gamma $ small enough such that 
		$$
		\frac{4(1+L)\gamma( \ln \frac{ \tilde z}{\delta}+4)}{ (\ln (  \tilde z/\delta)+10)}\leq \frac{\varepsilon_0}{100(1+L^2)^2},\quad\quad 
		\frac{4(1+L)\gamma}{ (\ln (  \tilde z/\delta)+10)}\leq\frac{\varepsilon_0}{100(1+L^2)^2} \omega(\tilde z),
		$$
		for $\tilde z>\delta$. Moreover, we have 
		\begin{align*}
			\frac{\omega(2\tilde z)-\omega(\tilde z)}{\tilde z}+\int_z^\infty\frac{\rho_T'(\alpha+z)}{\alpha}d\alpha\leq \frac{\gamma \ln 2}{10\tilde z}+\frac{\gamma}{\tilde z}\leq \frac{\lambda \omega(\delta)}{100(L+1)\tilde z}\leq \frac{\lambda \omega(\tilde z)}{100(L+1)\tilde z},
		\end{align*}
		as long as $\delta, \gamma$ are taken sufficiently small.
		We obtain \eqref{lar}. Then we complete the proof by setting 
		$$
		\rho(x)=\omega(C_0x).
		$$
	\end{proof}
	\subsection{ Uniqueness}
	\begin{theorem}
		Let $f\in L^\infty ([0,T]; W^{1,\infty})$ be a classical solution to the Muskat equation with initial data $f(0,x)=f_0(x)$. If there exists a modulus of
		continuity $\tilde \rho$ such that 
		\begin{align*}
			f_{x}(t, x)-f_{x}(t, y) \leq \tilde{\rho}(|x-y|), \quad \forall 0 \leq t \leq T, x \neq y \in \mathbb{R}.
		\end{align*}
		Then the solution $f$ is unique.
	\end{theorem}
	\begin{proof} The idea of the proof follows \cite{KC1}. 
		Let $f_1,f_2$ be two solutions with the same initial data. Denote $g=f_1-f_2$.  We have 
		\begin{align*}
			&\partial_t g-\int_{\mathbb{R}}\frac{E_\alpha g}{\langle\Delta_\alpha f_1\rangle^2}\frac{d\alpha}{\alpha}=\int_{\mathbb{R}}E_\alpha f_2\left(\frac{1}{\langle\Delta_\alpha f_2\rangle^2}-\frac{1}{\langle\Delta_\alpha f_1\rangle^2}\right)\frac{d\alpha}{\alpha}\\
			&\quad\quad \leq \int_{\mathbb{R}}|E_\alpha f_2||\Delta_\alpha g|\frac{d\alpha}{|\alpha|}.
		\end{align*}
		Take $x_t$ such that $g(t,x_t)=\sup_xg(t,x)$. Then 
		\begin{align*}
			\frac{d}{dt}g(x_t)+\int_{\mathbb{R}}\frac{\delta_\alpha g(x_t)}{\langle\Delta_\alpha f_1(x_t)\rangle^2}\frac{d\alpha}{\alpha^2}&\leq \int_{\mathbb{R}}|E_\alpha f_2(x_t)||\Delta_\alpha g(x_t)|\frac{d\alpha}{|\alpha|}\\
			&\leq\tilde \rho(\epsilon)\int_{|\alpha|\leq \epsilon}\delta_\alpha g(x_t)\frac{d\alpha}{\alpha^2}+\epsilon^{-1}\|f_2\|_{\dot W^{1,\infty}}\|g\|_{L^\infty}.
		\end{align*}
		Take $\epsilon$ small enough such that 
		\begin{align*}
			\tilde \rho(\epsilon)\leq \frac{1}{1+\|f_1\|_{\dot W^{1,\infty}}^2}.
		\end{align*}
		Then 
		\begin{align*}
			\frac{d}{dt}g(x_t)&\leq\epsilon^{-1}\|f_2\|_{\dot W^{1,\infty}}\|g\|_{L^\infty}.
		\end{align*}
		Integrate in time we obtain 
		\begin{align*}
			\sup_xg(t,x)\leq \epsilon^{-1}\|f_0\|_{\dot W^{1,\infty}}\int_0^t\|g(t)\|_{L^\infty}dt.
		\end{align*}
		A similar discussion holds if we replace $\sup g(x)$ by $\sup (-g(x))$. By Gronwall's inequality we have
		\begin{align*}
			\sup_{t\in[0,T]}\|g(t)\|_{L^\infty}\leq C\|g_0\|_{L^\infty}.
		\end{align*}
		This completes the proof.
	\end{proof}
	\section{Proof of Theorem \ref{thmhd}}
	We consider the Muskat equation in $d+1$-dimension.
	\begin{align}\label{hde}
		\partial_{t} f(t, x)=\int_{\mathbb{R}^d} \frac{\alpha \cdot \nabla f(t, x)-\delta_{\alpha} f(t, x)}{\left\langle\Delta_{\alpha} f(t, x)\right\rangle^{d+1}} \frac{d \alpha}{|\alpha|^{d+1}},
	\end{align}
	where $\Delta_\alpha f(t,x)=\frac{\delta_\alpha f(t,x)}{|\alpha|}$.
	Consider the initial data $f_0$ satisfying
	\begin{align*}
		\lim_{\varepsilon\to 0}\|\nabla f_0-\nabla f_0\ast \phi_\varepsilon\|_{L^\infty}\leq c_d,
	\end{align*}
	where $\phi_\varepsilon$ is the standard mollifier, and $c_d$ is a small constant depending only on dimension $d$. Then there exists $\eta>0$ such that 
	\begin{align}\label{inico}
		\|\nabla f_0-\nabla f_0\ast \phi_{\eta}\|_{L^\infty}\leq \frac{5}{4}c_d.
	\end{align} 
	
	Suppose for any $t\in[0,T]$, 
	\begin{align}
		&\|\nabla f(t,\cdot)-\nabla f_0\ast \phi_{\eta}\|_{L^\infty}\leq 2c_d.\label{hdc1}\\
		&|\nabla f(t,x)-\nabla f(t,y)|\leq \rho\left(\frac{|x-y|}{t}\right)\ \ \ \forall x\neq y\in\mathbb{R}^d,\label{hdcrho}
	\end{align}
	for some concave modulus function $\rho$ satisfying $\rho(\alpha)\leq C_0|\alpha|$. In the following, we will not distinguish $\rho(\alpha)$ and $\rho(|\alpha|)$.
	\begin{lemma}\label{lemhd}
		There exists $\tilde T=\tilde T(\|\nabla f_0\|_{L^\infty},d,\eta)>0$ such that 
		\begin{align*}
			\sup_{t\in[0,\min\{T,\tilde T\}]}\|\nabla f(t,\cdot)-\nabla f_0\ast \phi_{\eta}\|_{L^\infty}\leq \frac{3}{2}c_d.
		\end{align*}
	\end{lemma}
	\begin{proof}
		By \eqref{inico},	it suffices to find $\tilde T$ such that 
		\begin{align}\label{hdlip}
			\sup_{t\in[0,\min\{T,\tilde T\}]}\|\nabla f\|_{L^\infty}\leq \|\nabla f_0\|_{L^\infty}+\frac{c_d}{10}.
		\end{align}
		We follow the estimates in \cite{KC1}. Denote $f_1(t,x)=f(t,x)-f_0\ast\phi_\eta(x)$, $f_2(x)=f_0\ast\phi_\eta(x)$.
		For any $e\in \mathbb{S}^{d-1}$, denote $f_e=e\cdot \nabla f$.	Take directional derivative $\partial_e=e\cdot \nabla $ in \eqref{hde}, and denote $E_\alpha f(x)=\Delta_\alpha f(x)-\partial_{\hat{\alpha}}f(x)$ with $\hat{\alpha}=\frac{\alpha}{|\alpha|}$, we obtain
		\begin{align*}
			\partial_tf_e=&\int_{\mathbb{R}^d}\frac{\hat\alpha\cdot\nabla f_e}{\left\langle \Delta_\alpha f\right\rangle^{d+1}}\frac{d\alpha}{|\alpha|^d}-\int_{\mathbb{R}^d}\frac{\Delta_\alpha f_e}{\left\langle \Delta_\alpha f\right\rangle^{d+1}}\left(1+(d+1)\frac{E_\alpha f \Delta_\alpha f}{\langle \Delta_\alpha f\rangle^2}\right)\frac{d\alpha}{|\alpha|^d}.
		\end{align*}
		Let $x_{t,e}$, $\tilde x_{t,e}$ be such that $ f_e(t,x_{t,e})= \sup_xf_e(t,x)$ and  $f_e(t,\tilde x_{t,e})=\sup_x (- f_e(t,x))$.
		Then $\nabla f_e(t,x_{t,e})=0$ and $\Delta_\alpha f_e (t,x_{t,e})\geq 0$.  Denote $M_e(t)=f_e (t,x_{t,e})$ and $m_e(t)=f_e (t,\tilde x_{t,e})$. By \eqref{hdc1} we have 
		\begin{align}\label{ddd}
			\sup_{t\in[0,T]}(M_e(t)+m_e(t))\leq 2(\|\nabla f_0\|_{L^\infty}+2c_d).
		\end{align}
		Note that 
		\begin{align*}
			\int_{\mathbb{R}^d}&\frac{\Delta_\alpha f_e}{\left\langle \Delta_\alpha f\right\rangle^{d+1}}\frac{E_\alpha f_2 \Delta_\alpha f}{\langle \Delta_\alpha f\rangle^2}\frac{d\alpha}{|\alpha|^d}\\
			&\leq\lambda\int _{|\alpha|\leq \lambda}\frac{\Delta_\alpha f_e}{\left\langle \Delta_\alpha f\right\rangle^{d+1}}\frac{d\alpha}{|\alpha|^d}+\eta^{-2}\|\nabla f_0\|_{L^\infty}^2\int_{|\alpha|\leq \lambda}\frac{d\alpha}{|\alpha|^{d-1}}\\
			&\quad\quad+\|\nabla f_0\|_{L^\infty}(M_e+m_e)\int_{|\alpha|\geq \lambda}\frac{d\alpha}{|\alpha|^{d+1}}.
		\end{align*}
		We can take $\lambda$ small enough such that 
		\begin{align*}
			1+(d+1)\frac{E_\alpha f_1 \Delta_\alpha f}{\langle \Delta_\alpha f\rangle^2}-(d+1)\lambda\geq 	1-(d+1)(4c_d+\lambda) \geq 0.
		\end{align*}
		Then we obtain for any $t\in[0,T]$
		\begin{align*}
			\partial_t M_e&\leq (d+1) \eta^{-2}\lambda \|\nabla f_0\|_{L^\infty}^2+\lambda^{-1}\|\nabla f_0\|_{L^\infty}(M_e+m_e)\\
			&\overset{\eqref{ddd}}\leq (d+1) \eta^{-2}\lambda \|\nabla f_0\|_{L^\infty}^2+2\lambda^{-1}\|\nabla f_0\|_{L^\infty}
			(\|\nabla f_0\|_{L^\infty}+2c_d).
		\end{align*}
		Integrate the above estimate in time and take $\tilde T$ small enough, we obtain 
		\begin{align*}
			\sup_{t\in[0,\min\{\tilde T,T\}]} M_e(t)\leq M_e(0)+\frac{1}{10}c_d.
		\end{align*}
		We can estimate $m_e$ similarly. The above holds for any direction $e\in\mathbb{S}^{d-1}$. Then we obtain \eqref{hdlip}, which completes the proof.
	\end{proof}
	\subsection{ Modulus Estimates}
	Following Lemma \ref{strict}, if $\nabla f$ was to lose its modulus after time $T$, there must exist $x\neq y\in\mathbb{R}^d$ such that 
	\begin{align*}
		|\nabla f(T,x)-\nabla f(T,y)|= \rho\left(\frac{|x-y|}{T}\right).
	\end{align*}
	Hence there exists a direction $e\in\mathbb{S}^{d-1}$ such that 
	\begin{align*}
		f_e(T,x)-f_e(T,y)= \rho\left(\frac{|x-y|}{T}\right).
	\end{align*}
	
	Analogously as in Lemma \ref{lemmain}, we have 
	\begin{lemma}\label{lemmainhd}
		There exists a function $\rho$ such that if \eqref{hdc1}-\eqref{hdcrho} hold, and 
		\begin{align}\label{coeq1}
			f_e(T,x)-f_e(T,y)=\rho\left(\frac{|x-y|}{T}\right)	
		\end{align}
		for some $x\neq y\in\mathbb{R}^d$, and $e\in\mathbb{S}^{d-1}$, then there exists a constant $\tilde C$ independent in $T$ such that
		\begin{align*}
			\left.\frac{d}{dt}(f_e(t,x)-f_e(t,y))\right|_{t=T}<\left.\frac{d}{dt}\left(\rho\left(\frac{|x-y|}{t}\right)\right)\right|_{t=T}+\tilde C\rho\left(\frac{|x-y|}{T}\right).
		\end{align*}
	\end{lemma}
	\begin{proof}
		
		For simplicity, denote $f_2= f_0\ast \phi_\eta$ and $f_1=f-f_2$.
		Taking directional derivative $\partial_e=e\cdot\nabla$ in \eqref{hde} we obtain
		\begin{align*}
			\partial_{t} f_e(x)=\int_{\mathbb{R}^d} \frac{\hat{\alpha} \cdot \nabla f_e(x)}{\left\langle\Delta_{\alpha} f(x)\right\rangle^{d+1}} \frac{d\alpha}{|\alpha|^d}
			- \int _{\mathbb{R}^d}	R(x,\alpha){\Delta_{\alpha}  f_e(x)} \frac{d\alpha}{|\alpha|^d},
		\end{align*}
		where we denote $\hat \alpha=\frac{\alpha}{|\alpha|}$. Here
		\begin{align*}
			R(x,\alpha)&=\frac{1}{\langle\Delta_\alpha f\rangle^{d+1}}\left(1+(d+1)\frac{\Delta_{\alpha} fE_\alpha f(x)  }{\left\langle\Delta_{\alpha} f\right\rangle^{2}}\right)\\
			&=\frac{1}{\langle\Delta_\alpha f\rangle^{d+1}}\left(1+(d+1)\frac{\Delta_{\alpha} fE_\alpha f_1(x) }{\left\langle\Delta_{\alpha} f\right\rangle^{2}}\right)\\
			&\quad\quad+\frac{(d+1)\Delta_{\alpha} fE_\alpha f_2(x) }{\left\langle\Delta_{\alpha} f\right\rangle^{d+3}}\\
			&:=R_1+R_2.
		\end{align*} 
		Note that 
		\begin{align}\label{ellid}
			R_1\geq \frac{1-2(d+1)\|\nabla f_1\|_{L^\infty}}{(1+\|\nabla f\|_{L^\infty})^{d+1}}\geq \frac{1-4(d+1)c_d}{(1+\|\nabla f_0\|_{L^\infty}+c_d)^{d+1}}>0,
		\end{align}
		where we take $c_d<\frac{1}{5(d+1)}$.		We have the equation 
		\begin{align*}
			&\partial_{t} f_e(x)-\partial_{t} f_e(y)\\&=\int_{\mathbb{R}^d} \frac{\hat{\alpha} \cdot \nabla f_e(x)}{\left\langle\Delta_{\alpha} f(x)\right\rangle^{d+1}} - \frac{\hat{\alpha} \cdot \nabla f_e(y)}{\left\langle\Delta_{\alpha} f(y)\right\rangle^{d+1}} \frac{d\alpha}{|\alpha|^d}\\
			&\quad\quad- \int_{\mathbb{R}^d} \Delta_{\alpha}  f_e(x)	R_1(x,\alpha) - \Delta_{\alpha}  f_e(y)R_1(y,\alpha)	 \frac{d\alpha}{|\alpha|^d}\\
			&\quad\quad- \int_{\mathbb{R}^d} \Delta_{\alpha}  f_e(x)	R_2(x,\alpha) - \Delta_{\alpha}  f_e(y)R_2(y,\alpha)	 \frac{d\alpha}{|\alpha|^d}\\
			&:=\tilde Q_1+\tilde Q_2+\tilde Q_3.
		\end{align*}
		For simplicity, denote $z=x-y$. Similar as in \eqref{del}-\eqref{2de}, we have
		\begin{align*}
			\nabla f_e(T,x)=	\nabla f_e(T,y)=\nabla \rho_T(|z|)=\hat z\rho_T'(z).
		\end{align*}
		Hence 
		\begin{align*}
			&\int_{\mathbb{R}^d} \frac{\hat{\alpha} \cdot \nabla f_e(x)}{\left\langle\Delta_{\alpha} f(x)\right\rangle^{d+1}} \frac{d\alpha}{|\alpha|^d}-\int_{\mathbb{R}^d} \frac{\hat{\alpha} \cdot \nabla f_e(y)}{\left\langle\Delta_{\alpha} f(y)\right\rangle^{d+1}} \frac{d\alpha}{|\alpha|^d}\\
			&\quad\quad=\rho_T'(z)\int_{\mathbb{R}^d} \hat{\alpha} \cdot \hat z\left(\frac{1}{\left\langle\Delta_{\alpha} f(x)\right\rangle^{d+1}} -\frac{1}{\left\langle\Delta_{\alpha} f(y)\right\rangle^{d+1}}\right)\frac{d\alpha}{|\alpha|^d}.
		\end{align*}
		Using the same idea when we estimate $Q_1$ in Lemma \ref{lemmain}, we obtain 
		\begin{align}\label{Q1d}
			\tilde Q_1\leq \rho'_T(z)\left(\int_0^ {|z|}\rho_T\left(r\right)\frac{dr}{r}+|z|\int_ {|z|}^\infty\rho_T\left(r\right)\frac{dr}{r^2}\right).
		\end{align}
		Moreover, thanks to the ellipticity \eqref{ellid}, we have $R_1(x,\alpha)\geq  \lambda>0$. We can estimate $\tilde Q_2$ using ideas from \cite{Cameron2020}. More precisely, we denote
		\begin{equation*}
			G(x,y,\alpha)=\Delta_\alpha f_e(y)R_1(y,\alpha)-\Delta_\alpha f_e(x)R_1(x,\alpha).
		\end{equation*}
		We can  decompose $G(x,y,\alpha)$ in two different ways, namely
		\begin{equation*}
			\begin{aligned}
				G(x,y,\alpha)=(\Delta_\alpha f_e(y)-\Delta_\alpha f_e(x))R_1(y,\alpha)+\Delta_\alpha f_e(x)(R_1(y,\alpha)-R_1(x,\alpha))\\
				=(\Delta_\alpha f_e(y)-\Delta_\alpha f_e(x))R_1(x,\alpha)+\Delta_\alpha f_e(y)(R_1(y,\alpha)-R_1(x,\alpha)).
			\end{aligned}
		\end{equation*}
		Since
		\begin{equation*}
			\delta_\alpha f_e(y)-\delta_\alpha f_e(x)=-\rho_T(z)+f_e(x-\alpha)-f_e(y-\alpha)\leq 0.
		\end{equation*}
		By the ellipticity of $R_1$, we can see that 
		\begin{equation}\label{Glowb}
			\begin{aligned}
				G(x,y,\alpha)\leq& \lambda (\Delta_\alpha f_e(y)-\Delta_\alpha f_e(x))+(\Delta_\alpha f_e(x))_+(R_1(y,\alpha)-R_1(x,\alpha))_+\\
				&+(\Delta_\alpha f_e(y))_-(R_1(y,\alpha)-R_1(x,\alpha))_-.
			\end{aligned}
		\end{equation}
		Moreover, by \eqref{hdcrho} and  \eqref{coeq1} we have
		\begin{equation}\label{del1}
			\begin{aligned}
				&\delta_\alpha f_e(x)\geq \rho_T(|z|)-\rho_T(|z|+|\alpha|),\\
				&\delta_\alpha f_e(y)\leq -(\rho_T(|z|)-\rho_T(|z|+|\alpha|)).
			\end{aligned}
		\end{equation}
		Substituting this into \eqref{Glowb} gives that  
		\begin{equation*}
			\begin{aligned}
				&\int_{|\alpha|>|z|}G(x,y,\alpha)\frac{d\alpha}{|\alpha|^d} \\
				&\ \leq  \int_{|\alpha|>|z|}\lambda\big(\delta_\alpha f_e(y)-\delta_\alpha f_e(x)\big)+(\rho_T(|z|+|\alpha|)-\rho_T(|z|))|R_1(x,\alpha)-R_1(y,\alpha)|\frac{d\alpha}{|\alpha|^{d+1}}.
			\end{aligned}
		\end{equation*}
		Denote $L=2\|\nabla f_0\|_{L^\infty}+4c_d$, then we can prove that
		\begin{equation}\label{difR1}
			|R_1(x,\alpha)-R_1	(y,\alpha)|\leq C_d(1+L)\rho_T(|z|).
		\end{equation}
		Hence
		\begin{equation}\label{lar2}
			\begin{aligned}
				\int_{|\alpha|>|z|}G(x,y,\alpha)\frac{d\alpha}{|\alpha|^d} \leq  &\lambda\int_{|\alpha|>|z|}\frac{\Delta_\alpha f_e(y)-\Delta_\alpha f_e(x)}{|\alpha|^d}d\alpha\\
				&+C_d(1+L)\rho_T(|z|)\int_{|z|}^{\infty}\frac{\rho_T(|z|+r)-\rho_T(|z|)}{r^2}dr.
			\end{aligned}
		\end{equation}
		For the other case $|\alpha|<|z|$, by adding and
		subtracting a linear term, we have that 
		\begin{equation*}
			\begin{aligned}
				G(x,y,\alpha)=&(\Delta_\alpha f_e(y)-\rho'_T(z)\hat{z}\cdot\hat{\alpha})R_1(y,\alpha)-(\Delta_\alpha f_e(x)-\rho'_T(z)\hat{z}\cdot\hat{\alpha})R_1(x,\alpha)\\
				&+\rho'_T(|z|)\hat{z}\cdot\hat{\alpha}(R_1(y,\alpha)-R_1(x,\alpha)).
			\end{aligned}
		\end{equation*}
		By the concavity of $\rho_T$, we have 
		\begin{equation*}
			\rho_T(|h+\xi|)-\rho_T(|\xi|)\leq \rho'_T(|\xi|)(|h+\xi|-|\xi|)\leq \rho'_T(|\xi|)\hat{\xi}\cdot h+\frac{\rho'_T(|\xi|)}{|\xi|}|h|^2.
		\end{equation*}
		By similar methods in \eqref{Glowb}, and use $\eqref{del1}$, we can obtain that
		\begin{equation*}
			\begin{aligned}
				G(x,y,\alpha)\leq&\lambda(\Delta_\alpha f_e(y)-\Delta_\alpha f_e(x))+(\Delta_\alpha f_e(x)-\rho'_T(z)\hat{z}\cdot\hat{\alpha})_+(R_1(x,\alpha)-R_1(y,\alpha))_+\\
				&+(\Delta_\alpha f_e(y)-\rho'_T(z)\hat{z}\cdot\hat{\alpha})_-(R_1(x,\alpha)-R_1(y,\alpha))_-\\
				&+\rho'_T(|z|)\hat{z}\cdot \hat \alpha(R_1(y,\alpha)-R_1(x,\alpha))\\
				\leq &\lambda(\Delta_\alpha f_e(y)-\Delta_\alpha f_e(x))+\rho'_T(|z|)\hat{z}\cdot\hat{\alpha}(R_1(y,\alpha)-R_1(x,\alpha))\\
				&+\frac{\rho'_T(|z|)}{|z|}|\alpha||R_1(x,\alpha)-R_1(y,\alpha)|.
			\end{aligned}
		\end{equation*}
		%
		Combining this with \eqref{difR1} to get that
		\begin{equation}\label{sma2}
			\begin{aligned}
				\int_{|\alpha|<|z|}G(x,y,\alpha)\frac{d\alpha}{|\alpha|^d}\leq& \lambda\int_{|\alpha|<|z|}\frac{\Delta_\alpha f_e(y)-\Delta_\alpha f_e(x)}{|\alpha|^d}d\alpha\\
				& +2C_d(1+L)\rho'_T(|z|)\int_0^{|z|}\frac{\rho_T(r)}{r}dr.
			\end{aligned}
		\end{equation}
		So we finally get from \eqref{lar2} and \eqref{sma2} that
		\begin{equation*}
			\begin{aligned}
				&\int_{\mathbb{R}^d}G(x,y,\alpha)\frac{d\alpha}{|\alpha|^d}- \lambda\int_{\mathbb{R}^d}\frac{\Delta_\alpha f_e(y)-\Delta_\alpha f_e(x)}{|\alpha|^d}d\alpha \\
				&\leq 2C_d(1+L)\rho'_T(|z|)\int_0^{|z|}\frac{\rho_T(r)}{r}dr+C_d(1+L)\rho_T(|z|)\int_{|z|}^{\infty}\frac{\rho_T(|z|+r)-\rho_T(|z|)}{r^2}dr.
			\end{aligned}
		\end{equation*}
		By  in \cite[Lemma 5.2]{Cameron2020}, we can see that
		\begin{equation*}
			\begin{aligned}
				\int_{\mathbb{R}^d}\frac{\Delta_\alpha f_e(y)-\Delta_\alpha f_e(x)}{|\alpha|^d}d\alpha\leq&-\mathbf{c}_d \int_0^{|z|}\frac{\delta_r \rho_T(|z|)+\delta_{-r}\rho_T(|z|)}{r^2}dr\\
				&+\mathbf{c}_d\int_{|z|}^{\infty}\frac{\rho_T(r+|z|)-\rho_T(r-|z|)-2\rho_T(z)}{r^2}dr.
			\end{aligned}
		\end{equation*}
		Hence we conclude that 
		\begin{equation}\label{Q2}
			\begin{aligned}
				\tilde Q_2\leq& C_d(1+L)\rho'_T(|z|)\int_0^{|z|}\frac{\rho_T(r)}{r}dr+C_d(1+L)\rho_T(|z|)\int_{|z|}^{\infty}\frac{\rho_T(|z|+r)-\rho_T(|z|)}{r^2}dr\\
				&-\mathbf{c}_d\lambda\int_0^{|z|}\frac{\delta_r \rho_T(|z|)+\delta_{-r}\rho_T(|z|)}{r^2}dr+\mathbf{c}_d\lambda\int_{|z|}^{\infty}\frac{\rho_T(r+|z|)-\rho_T(r-|z|)-2\rho_T(z)}{r^2}dr.
			\end{aligned}
		\end{equation}
		It remains to estimate $\tilde Q_3$, we  have 
		\begin{align*}
			\tilde Q_3
			&=\int_{|\alpha|\leq |z|} \left({\delta_{\alpha}  f_e(y)}	R_2(y,\alpha) - {\delta_{\alpha}  f_e(x)}	R_2(x,\alpha)\right) \frac{d\alpha}{|\alpha|^{d+1}}\\
			&\quad\quad\quad+\int_{|\alpha|\geq |z|} \left({\delta_{\alpha}  f_e(y)}	R_2(y,\alpha) - {\delta_{\alpha}  f_e(x)}	R_2(x,\alpha)\right) \frac{d\alpha}{|\alpha|^{d+1}}\\
			&:=\tilde Q_{3,1}+\tilde Q_{3,2}.
		\end{align*}
		Observe that
		\begin{align*}
			|R_2{(x,\alpha)}|\leq (d+1)|E_\alpha f_2(x)|.
		\end{align*} 
		Recall the definition of $f_2$, we have 
		\begin{align*}
			\frac{|R_2{(x,\alpha)}|}{|\alpha|^\theta}\lesssim	\frac{|E_\alpha f_2(x)|}{|\alpha|^\theta}\lesssim \|f_2\|_{\dot C^{1,\theta}}\lesssim \eta^{-\theta}\|\nabla f_0\|_{L^\infty},\ \ \ \forall \theta\in(0,1].
		\end{align*}
		Hence,
		\begin{align}\label{Q21}
			\tilde Q_{3,1}
			&\leq C_d\eta^{-1}L\int_0^{ |z|}\frac{\rho_T(r)}{r}dr.
		\end{align}
		On the other hand, we have 
		\begin{align*}
			\tilde Q_{3,2}	&=\int_{|\alpha|\geq |z|} \left({\delta_{\alpha}  f_e(y)}- {\delta_{\alpha}  f_e(x)}\right)	R_2(y,\alpha)\frac{d\alpha}{|\alpha|^{d+1}}\\
			&\quad\quad\quad+\int_{|\alpha|\geq |z|} {\delta_{\alpha}  f_e(x)}\left(	R_2(y,\alpha)-R_2(x,\alpha)\right)\frac{d\alpha}{|\alpha|^{d+1}}\\
			&=\tilde{Q}_{3,2,1}+\tilde{Q}_{3,2,2}.
		\end{align*}
		By \eqref{hdcrho}, it is easy to check that
		\begin{align*}
			|\delta_\alpha f_e(x)|\leq \rho_T(|\alpha|),\ \ \ 
			\left|{\delta_{\alpha}  f_e(y)}- {\delta_{\alpha}  f_e(x)}\right|\leq 2\rho_T(|z|),\ \ 
			|\Delta_\alpha f(y)-\Delta_\alpha f(x)|\leq 2\rho_T(|z|).
		\end{align*}
		Moreover, recall that $f_2=f_0\ast \phi_\eta$, hence 
		\begin{align*}
			&|\hat \alpha\cdot \nabla f_2(h)-\Delta_\alpha f_2(h)|\leq \min\{1,\eta^{-1}{|\alpha|}\}\|\nabla f_0\|_{L^\infty},\\
			&|(\hat \alpha\cdot \nabla f_2-\Delta_\alpha f_2)(x)-(\hat \alpha\cdot \nabla f_2-\Delta_\alpha f_2)(y)|\leq \min\{1,\eta^{-1}|z|\} \min\{1,\eta^{-1}{|\alpha|}\}\|\nabla f_0\|_{L^\infty}.
		\end{align*}
		Then we obtain 
		\begin{align*}
			&|R_2(x,\alpha)|\leq C_d (1+L)\min\{1,\eta^{-1}{|\alpha|}\},\\	&\left|	R_2(y,\alpha)-R_2(x,\alpha)\right|\leq C_d (1+L)\left(\rho_T(z)+\min\{1,{\eta}^{-1}{|z|}\}\right)\min\{1,\eta^{-1}{|\alpha|}\}.
		\end{align*}
		Hence 
		\begin{align*}
			|\tilde Q_{3,2,1}|	&\leq C_d(1+L)\rho_T(|z|)\int_{|z|}^\infty \min\{1,\eta^{-1}r\}\frac{dr}{r^2},\\
			|\tilde Q_{3,2,2}|	&\leq C_d(1+L)(\rho_T(|z|)+\min\{1,\eta^{-1}|z|\})\int_{|z|}^\infty \rho_T(r)\min\{1,\eta^{-1}r\}\frac{dr}{r^2}.
		\end{align*}
		We conclude from this and \eqref{Q21} that 
		\begin{equation}\label{Q3}
			\begin{aligned}
				|\tilde Q_3|\leq &C_d \eta^{-1}L\int_0^{|z|}\frac{\rho_T(r)}{r}dr+ C_d(1+L)^2\rho_T(|z|)\int_{|z|}^\infty \min\{1,\eta^{-1}r\}\frac{dr}{r^2}\\
				&+ C_d(1+L)\min\{1,\eta^{-1}|z|\}\int_{|z|}^\infty \rho_T(r)\min\{1,\eta^{-1}r\}\frac{dr}{r^2}.
			\end{aligned}	
		\end{equation}
		We conclude from \eqref{Q1d}, \eqref{Q2} and \eqref{Q3} that
		\begin{equation}\label{hdre}
			\begin{aligned}
				&\left.\frac{d}{dt}(f_e(t,x)-f_e(t,y))\right|_{t=T}\\
				&\leq C_d(1+L)\rho'_T(|{z}|)\left(\int_{0}^{|{z}|}\rho_T(r)\frac{dr}{r}+|{z}|\int_{|{z}|}^\infty\rho_T(r)\frac{dr}{r^2}\right)\\
				&\ \ \ +C_d(1+L)\rho_T(|z|)\int_{|{z}|}^\infty\frac{\rho_T(r+|{z}|)-\rho_T(|z|)}{r^2}dr-\mathbf{c}_d\lambda\int_0^{|{z}|}\frac{\delta_r\rho_T(|z|)+\delta_{-r}\rho_T(|z|)}{r^2}dr\\
				&\ \ \ +\mathbf{c}_d\lambda\int_{|{z}|}^{\infty}\frac{\rho_T(r+|{z}|)-\rho_T(r-|{z}|)-2\rho_T(|z|)}{r^2}dr\\
				&\ \ \ +C_d \eta^{-1}L\int_0^{|z|}\frac{\rho_T(r)}{r}dr+ C_d(1+L)^2\rho_T(|z|)\int_{|z|}^\infty \min\{1,\eta^{-1}r\}\frac{dr}{r^2}\\
				&\ \ \ + C_d(1+L)\min\{1,\eta^{-1}|z|\}\int_{|z|}^\infty \rho_T(r)\min\{1,\eta^{-1}r\}\frac{dr}{r^2}.
			\end{aligned}
		\end{equation}
		We denote $\tilde{z}=\frac{C_0z}{T}$, and $\omega(|\tilde{z}|)=\rho_T(|z|)$, where the constant $C_0$ will be fixed later. Then we have 
		\begin{align*}
			\rho_T'(|z|)=\frac{C_0}{T}\omega'(|\tilde{z}|),\ \  \ 
			\rho''_T(|z|)=\frac{C_0^2}{T^2}\omega''(|\tilde{z}|).
		\end{align*}
		We take $\omega(\alpha)$ as follows, where $\delta,\gamma$ will be fixed later,
		\begin{equation*}
			\begin{aligned}
				&	\omega(\alpha)=\alpha-\frac{4}{3}\alpha^{\frac{3}{2}},\ \alpha\leq\delta,\\
				&	\omega'(\alpha)=\frac{\gamma}{\alpha(\ln(\frac{\alpha}{\delta}+10))},\ \alpha\geq\delta.
			\end{aligned}
		\end{equation*}
		We mainly discuss the last two lines in the right hand side of \eqref{hdre}, which are new terms compared with \eqref{final}.  For simplicity, denote 
		\begin{align*}
			&B_1=-	C_d \eta^{-1}L\int_0^{|z|}\frac{\rho_T(r)}{r}dr,\ \ \ \ B_2=C_d(1+L)^2\rho_T(|z|)\int_{|z|}^\infty \min\{1,{\eta}^{-1}r\}\frac{dr}{r^2},\\
			&B_3= C_d(1+L)\min\{1,\eta^{-1}|z|\}\int_{|z|}^\infty \rho_T(r)\min\{1,{\eta}^{-1}r\}\frac{dr}{r^2}.
		\end{align*}
		Our aim is to prove that for any $0<\epsilon_1\ll 1$, there exist suitable constants $\delta,\gamma, \tilde C$ such that 
		\begin{equation}\label{B123}
			B_1+B_2+B_3\leq \begin{cases}
				-	 \epsilon_1 |z|\rho_T''(|z|)+\tilde C \rho_T(|z|),\ \ \text{if}\ |\tilde z|\leq \delta,\\
				\epsilon_1 \frac{\rho_T(|z|)}{|z|}+\tilde C \rho_T(|z|),\quad\quad \ \text{if}\ |\tilde z|> \delta.
			\end{cases}
		\end{equation}
		Using the fact that  $\omega(|\tilde{z}|)\leq 2L$, we have  $\tilde{z}\leq \omega^{-1}(2L)$. Take $C_0\geq \frac{\omega^{-1}(2L)}{2L}$, by concavity we obtain that 
		\begin{equation*}
			\frac{|\tilde{z}|}{\omega(|\tilde{z}|)}\leq\frac{\omega^{-1}(2L)}{2L}\leq C_0.
		\end{equation*}
		Hence 
		\begin{align*}
			B_1=C_d\eta^{-1} L\int_{0}^{|\tilde z|} \frac{\omega(r)}{r}dr\leq   C_d\eta^{-1}L|\tilde z|\leq \tilde C_1\omega(|\tilde z|),
		\end{align*}
		where $\tilde C_1=C_0C_d\eta^{-1}L$.\\
		For $B_2$ and $B_3$, we consider $|\tilde z|\leq \delta$ and $|\tilde z|\geq \delta$ separately. If $|\tilde z|\leq \delta$, then 
		\begin{align*}
			B_2&\leq C_d \eta^{-1}(1+L)^2 \rho_T(|z|)\left(1+\mathbf{1}_{|z|\leq\eta}\ln \frac{\eta}{|z|}\right)\\
			&= C_d \eta^{-1}(1+L)^2 \omega(|\tilde z|)\left(1+\mathbf{1}_{|z|\leq\eta}\ln \frac{\eta}{|z|}\right)\\
			&\leq -\frac{\epsilon_1}{10} \frac{C_0|\tilde z|\omega''(|\tilde z|)}{T}=-\frac{\epsilon_1}{10}|z|\rho_T''(|z|),
		\end{align*}
		this follows by taking $\delta$ small enough such that $C_d \eta^{-2}(1+L)^2\left(1+\mathbf{1}_{|z|\leq\eta}\ln \frac{\eta}{|z|}\right)\leq \left(\frac{C_0}{T|z|}\right)^\frac{1}{2}=-\frac{\epsilon_1}{10}\frac{C_0\omega''(|\tilde z|)}{T}$ holds for any $|\tilde z|\leq \delta$. Moreover, by concavity, we have $\frac{\rho_T(r)}{r}\leq \frac{\rho_T(|z|)}{|z|}$, hence 
		\begin{align*}
			\int_{|z|}^\infty\rho_T(r)\min\{1,{\eta}^{-1}r\}\frac{dr}{r^2}&\leq \int_{|z|}^{\max\{|z|,\eta\}}\eta^{-1}\frac{\rho_T(|z|)}{|z|}{dr}+L\int_{\max\{|z|,\eta\}}^\infty\frac{dr}{r^2}\\
			&\leq \frac{\rho_T(z)}{|z|}+L\eta^{-1}.
		\end{align*}
		This yields that 
		\begin{align*}
			B_3&\leq C_d(1+L)^2\eta^{-1}|z|\left(\frac{\rho_T(z)}{|z|}+\eta^{-1}\right)\leq C_d(1+L)^2\eta^{-1}\rho_T(z)+C_d(1+L)^2\eta^{-2}|z|\\
			&\leq \tilde C_2\rho_T(z)-\frac{\epsilon_1}{10} \frac{C_0|\tilde z|\omega''(|\tilde z|)}{T}=\tilde C_2\rho_T(z)-\frac{\epsilon_1}{10} |z|\rho_T''(|z|),
		\end{align*}where $\tilde C_2=C_d(1+L)^2\eta^{-1}$.\\
		On the other hand, if $|\tilde z|>\delta$, we have 
		\begin{align*}
			&\int_{|z|}^\infty \min\{1,\eta^{-1 }r\}\frac{dr}{r^2}\leq\eta^{-1}\int_{|z|}^{M|z|}\frac{dr}{r}+\int_{M|z|}^\infty\frac{dr}{r^2}
			\leq \eta^{-1}\ln M+\frac{1}{M|z|},
		\end{align*}
		for any $M>1$.
		Take $M=\frac{10C_d (1+L)^2}{\epsilon_1}$, then
		\begin{align*}
			B_2\leq C_d (1+L)^2\rho_T(|z|)\left(\eta^{-1}\ln M+\frac{1}{M|z|}\right)\leq \tilde C_3 \rho_T(|z|)+\frac{\epsilon_1}{10}\frac{\rho_T(|z|)}{|z|},
		\end{align*}
		where $\tilde C_3=C_d (1+L)^2\eta^{-1}\log M$.
		Moreover, it is easy to check that 
		\begin{align*}
			\int_{|z|}^\infty \rho_T(r)\frac{dr}{r^2}&=\int_{|z|}^\infty \left(\rho_T(|z|)+\int_{|z|}^r \rho_T'(\xi)d\xi\right)\frac{dr}{r^2}\\
			&=\frac{\rho_T(|z|)}{|z|}+\int_{|z|}^\infty \rho_T'(\xi)  \frac{d\xi}{\xi}.
		\end{align*}
		By definition, 
		\begin{align*}
			\int_{|z|}^\infty\rho_T'(\xi)\frac{d\xi}{\xi}=\frac{C_0}{T}\int_{|\tilde z|}^\infty \frac{\omega'(r)}{r}dr=\frac{C_0}{T}\int_{|\tilde z|}^\infty \frac{\gamma}{r^2(\log(\frac{r}{\delta}+10))}dr\leq \gamma \frac{C_0}{T|\tilde z|}=\frac{\gamma}{|z|}.
		\end{align*}
		Hence one has 
		\begin{align*}
			B_3\leq C_d \eta^{-1}(1+L)^2\left(\rho_T(|z|)+{\gamma}\right) \lesssim \tilde C_4 \omega(|\tilde z|),\ \ \text{where}\ \tilde C_4= 2{C_d \eta^{-1}(1+L)^2},
		\end{align*}
		where we take $\gamma<\omega(\delta)$. Then we obtain \eqref{B123} by taking $\tilde C=\sum_{k=1}^4\tilde C_k$. 
		
		Other terms in \eqref{hdre}, except $B_1,B_2,B_3$, have been calculated in detail in the proof of Theorem \ref{thm}(see also \cite{Cameron2020}). For simplicity, we denote 
		\begin{align*}
			&A_1=C_d(1+L)\rho'_T(|{z}|)\left(\int_{0}^{|{z}|}\rho_T(r)\frac{dr}{r}+|{z}|\int_{|{z}|}^\infty\rho_T(r)\frac{dr}{r^2}\right),\\
			&A_2=C_d(1+L)\rho_T(|z|)\int_{|{z}|}^\infty\frac{\rho_T(r+|{z}|)-\rho_T(|z|)}{r^2}dr,
		\end{align*}
		and diffusion term
		\begin{align*}
			D=&-\mathbf{c}_d\lambda\int_0^{|{z}|}\frac{\delta_r\rho_T(|z|)+\delta_{-r}\rho_T(|z|)}{r^2}dr\\
			&+\mathbf{c}_d\lambda\int_{|{z}|}^{\infty}\frac{\rho_T(r+|{z}|)-\rho_T(r-|{z}|)-2\rho_T(|z|)}{r^2}dr.
		\end{align*}
		Without repeated calculation, we claim that 
		\begin{align}\label{Dif}
			D\leq \begin{cases}
				\frac{\mathbf{c}_d\lambda}{2}\frac{C_0}{T}|\tilde z|\omega''(|\tilde z|)= \frac{\mathbf{c}_d\lambda}{2} |z|\rho_T''(|z|) ,\ \ \text{if}\ |\tilde z|\leq \delta,\\
				-\frac{\mathbf{c}_d\lambda}{2}\frac{C_0}{T}\frac{\omega(|\tilde z|)}{|\tilde z|}= -\frac{\mathbf{c}_d\lambda}{2} \frac{\rho_T(|z|)}{|z|},\quad\quad \ \text{if}\ |\tilde z|> \delta.
			\end{cases}
		\end{align}And, 
		for any $\epsilon_2>0$, by suitable choice of $\delta$ and $\gamma$, there holds 
		\begin{align}\label{A12}
			A_1+A_2\leq \begin{cases}
				-\epsilon_2 |z|\rho_T''(|z|),\ \ \text{if}\ |\tilde z|\leq \delta,\\
				\epsilon_2 \frac{\rho_T(|z|)}{|z|},\quad\quad \ \text{if}\ |\tilde z|> \delta.
			\end{cases}
		\end{align}
		To finish the proof, we need to show that 
		\begin{align}\label{ABD}
			& A_1+A_2+B_1+B_2+B_3+D< \frac{d}{dt}\left.\left(\rho_t(|z|)\right)\right|_{t=T}+\tilde C\rho_T(|z|)=-\frac{\tilde{z}}{T}\omega'(|\tilde{z}|)+\tilde{C}w(|\tilde{z}|).
		\end{align}
		By the definition of $\omega$, it is easy to check that 
		\begin{align*}
			\frac{\tilde z}{T}\omega'(|\tilde z|)\leq \begin{cases}
				-\frac{\mathbf{c}_d\lambda}{100} \frac{C_0}{T}|\tilde z|\omega''(|\tilde z|),\ \ \text{if}\ |\tilde z|\leq \delta,\\
				\frac{\mathbf{c}_d\lambda}{100} \frac{C_0}{T}\frac{\omega(|\tilde z|)}{|\tilde z|},\quad\quad \ \text{if}\ |\tilde z|> \delta,
			\end{cases}
		\end{align*}
		assuming $\delta,\gamma$ sufficiently small. 
		Take $\epsilon_1=\epsilon_2=\frac{\mathbf{c}_d\lambda}{100}$, by \eqref{B123}, \eqref{Dif} and \eqref{A12}, we can take $\delta,\gamma $ small enough such that 
		\begin{align*}
			& A_1+A_2+B_1+B_2+B_3+D+\frac{\tilde{z}}{T}\omega'(|\tilde{z}|)\\
			&< \begin{cases}
				\frac{\mathbf{c}_d\lambda}{10} \frac{C_0}{T}|\tilde z|\omega''(|\tilde z|)+\tilde C \omega(|\tilde z|)
				\ \leq \tilde C\rho_T(|z|),\ \ \ \text{if}\ |\tilde z|\leq \delta,\\
				-\frac{\mathbf{c}_d\lambda}{10} \frac{C_0}{T}\frac{\omega(|\tilde z|)}{|\tilde z|}+\tilde C \omega(|\tilde z|)\leq \tilde C\rho_T(|z|),\quad\quad \ \text{if}\ |\tilde z|> \delta.
			\end{cases}
		\end{align*}
		This implies \eqref{ABD}. Then we complete the proof.

	\end{proof}
	\begin{proof}[Proof of Theorem \ref{thmhd}]
		Now we give a proof of the existence of solution in high dimensions by bootstrap. Define the smooth initial data 
		$$
		f_{0,\nu}= (f_0\chi_{\nu^{-1}})\ast \phi_{\nu},\ \ \ \nu\in(0,1).
		$$
		Denote $f_\nu$ the unique classical solution to \eqref{hde} with initial data $f_{0,\nu}$, and set $\tilde{T}$ as defined in  Lemma \ref{lemhd}. We define \begin{align*}
			T_0=&\sup\left\{T\in [0,\tilde{T}]: \sup_{t\in[0,T]}\|\nabla f_\nu-\nabla f_{0,\nu}*\phi_\eta\|_{L^\infty}<2c_d,\right.\\
			&\quad\quad\quad\left.|\nabla f_\nu(t,x)-\nabla f_\nu(t,y)|\leq\rho\left(\frac{|x-y|}{t}\right)e^{\tilde C t},\ \forall x\neq y\in\mathbb{R}^d, t\in[0,T]. \right\}
		\end{align*} Then by Lemma \ref{lemhd} and Lemma \ref{lemmainhd}, we can prove that $T_0=\tilde{T}$ following the proof of Theorem \ref{thm}.
		By compactness argument, we obtain the result by taking $\nu\rightarrow 0$.
	\end{proof}
	
	\subsection{ Uniqueness}
	
	First we define
	\begin{equation*}
		\rho(g,\epsilon) := \sup\{r \geq 0: |g(x+h) -g(x)|\leq \epsilon, \forall x\in \mathbb{R}^d, |h|\leq r\}.
	\end{equation*}
	Our main theorem for uniqueness in high dimension is as follows. 
	\begin{theorem}
		Let $f_1$, $f_2$ be two classical solutions of the Muskat equation \eqref{hde} with $\|\nabla f_i\|_{L^\infty}\leq L$, $\rho(\nabla f_i(t),\frac{2}{d+1})\geq \rho_0$, and 
		$$
		\lim_{|x|\rightarrow \infty}f_i(t,x)=0
		$$
		uniformly in $t$ for $i=1,2$. Then $M(t) = ||f_1-f_2||_{L^\infty}(t)$ satisfies 
		\begin{equation*}
			M'(t) \leq C_d(1+L)\rho_0^{-1}M(t),
		\end{equation*}
		where $C_d$ is a constant only depends on dimension. In particular, 
		\begin{equation}\label{uni}
			||f_1 - f_2||_{L^\infty}(t) \leq ||f_1-f_2||_{L^\infty}(0)\exp\left((C_d(1+L)\rho_0^{-1})t\right).
		\end{equation}
	\end{theorem}
	\begin{proof}
		Without loss of generality, assume $M(t)=\sup_{x}(f_1(t,x)-f_2(t,x))$.  There exists a point $x_t$ such that 
		$$
		M(t)=f_1(t,x_t)-f_2(t,x_t), 
		$$
		then we will get
		\begin{align}\label{hddd}
			\Delta_\alpha (f_1-f_2)(x_t) \geq 0, \qquad \nabla_x (f_1-f_2)(x_t) =0.
		\end{align}	
		We deduce from the equation \eqref{hde} that
		\begin{align*}
			&	\partial_t f_1(t,x_t) - \partial_t f_2(t,x_t)\\ &\ \ =\int_{\mathbb{R}^d} \left(\frac{\hat \alpha\cdot\nabla f_1(x_t) -\Delta_\alpha f_1(x_t)  }{\langle\Delta_\alpha f_1(x_t)\rangle^{{d+1}}} - \frac{\hat \alpha\cdot\nabla f_2(x_t) -\Delta_\alpha f_2(x_t) }{\langle\Delta_\alpha f_2(x_t)\rangle^{d+1}}\right)\frac{d\alpha}{|\alpha|^{d}}.
		\end{align*}
		We first claim that for any $|\alpha|<\rho_0$, there holds 
		\begin{align*}
			\frac{\hat \alpha\cdot\nabla_x f_1(x_t) -\Delta_\alpha f_1(x_t)  }{\langle\Delta_\alpha f_1(x_t)\rangle^{{d+1}}} - \frac{\hat \alpha\cdot\nabla_x f_2(x_t) -\Delta_\alpha f_2(x_t) }{\langle\Delta_\alpha f_2(x_t)\rangle^{d+1}}\leq 0.
		\end{align*}
		To see this, we  consider 
		$$
		N(a, b):=\frac{b-a}{\langle a\rangle^{d+1}}.
		$$
		We take derivative in $a$ and get
		\begin{align*}
			\partial_a N(a, b)=\frac{-1}{\langle a\rangle^{d+1}}\left(1-(d+1)\frac{a(a-b)}{\langle a\rangle^2}\right)\geq \frac{-1}{\langle a \rangle^{d+1}}\left(1-\frac{d+1}{2}|a -b|\right).
		\end{align*}
		Thus if $|a-b|\leq \frac{2}{d+1}$, then $N(a,b)$ is non-increasing in $a$. Note that when $|\alpha|\leq \rho_0$, we have $|\hat \alpha\cdot\nabla_x f_1(x_t) -\Delta_\alpha f_1(x_t) |\leq \frac{2}{d+1}$. Then  by \eqref{hddd} we obtain 
		\begin{align*}
			&\frac{\hat \alpha\cdot\nabla_x f_1(x_t) -\Delta_\alpha f_1(x_t)  }{\langle\Delta_\alpha f_1(x_t)\rangle^{{d+1}}} - \frac{\hat \alpha\cdot\nabla_x f_2(x_t) -\Delta_\alpha f_2(x_t) }{\langle\Delta_\alpha f_2(x_t)\rangle^{d+1}}\\
			&=N(\Delta_\alpha f_1(x_t),\hat \alpha\cdot\nabla_x f_1(x_t))-N(\Delta_\alpha f_2(x_t),\hat \alpha\cdot\nabla_x f_2(x_t))\leq 0.
		\end{align*}
		It remains to consider  $|\alpha|>\rho_0$. We have 
		\begin{align*}
			&	\partial_t f_1(t,x_t) - \partial_t f_2(t,x_t)\\ &\leq \int_{|\alpha|>\rho_0} \left(\frac{\hat \alpha\cdot\nabla f_1(x_t) -\Delta_\alpha f_1(x_t)  }{\langle\Delta_\alpha f_1(x_t)\rangle^{{d}}} - \frac{\hat \alpha\cdot\nabla f_2(x_t) -\Delta_\alpha f_2(x_t) }{\langle\Delta_\alpha f_2(x_t)\rangle^{d+1}}\right)\frac{d\alpha}{|\alpha|^{d}}\\
			&\leq  \int_{|\alpha|>\rho_0} \hat \alpha\cdot\nabla f_1(x_t)\left(\frac{1  }{\langle\Delta_\alpha f_1(x_t)\rangle^{{d}}}-\frac{1  }{\langle\Delta_\alpha f_2(x_t)\rangle^{{d}}}\right)\frac{d\alpha}{|\alpha|^{d}} \\&\ \ \ +\int_{|\alpha|>\rho_0} \left|\frac{\Delta_\alpha f_1(x_t) }{\langle\Delta_\alpha f_1(x_t)\rangle^{d+1}}-\frac{ \Delta_\alpha f_2(x_t) }{\langle\Delta_\alpha f_2(x_t)\rangle^{d+1}}\right|\frac{d\alpha}{|\alpha|^{d}}.
		\end{align*}
		By the elementary inequality
		\begin{align*}
			\left|\frac{1}{\langle a\rangle^d}-	\frac{1}{\langle b\rangle^d}\right|+	\left|\frac{a}{\langle a\rangle^{d+1}}-	\frac{a}{\langle a\rangle^{d+1}}\right|\lesssim _d |a-b|,
		\end{align*}
		we obtain that 
		\begin{align*}
			\partial_t f_1(t,x_t) - \partial_t f_2(t,x_t)& \leq C_d  M(t) (1+L)\int_{|\alpha|>\rho_0}\frac{d\alpha}{|\alpha|^{d+1}}\\
			&\leq C_d  \rho_0^{-1}(1+L)M(t).
		\end{align*}
		This yields that 
		\begin{align*}
			\frac{d}{dt } M(t)\leq C_d  \rho_0^{-1}(1+L)M(t).
		\end{align*}
		By Gronwall's inequality, we obtain \eqref{uni}. This completes the proof.
	\end{proof}

	\section{Appendix}
	\begin{lemma}\label{lem2de}
		Let $f$ be a classical solution to the Muskat equation with initial data satisfying $f'_0\in C_c^\infty(\mathbb{R}^d)$. Suppose $\sup_{t\in[0,T]}\|f(t)\|_{C^{3}}\leq R$. Then there exists $\varepsilon_1=\varepsilon_1(R,T)>0$ such that $\operatorname{supp}f_0\subset B_{\varepsilon_1^{-1}}(0)$ and for any $0<\varepsilon\leq \varepsilon_1$ 
		\begin{align*}
			\sup_{t\in[0,T]}(\|\partial_x(\tilde \chi_{\varepsilon} f(t))\|_{L^\infty}+	\|\partial_{xx}(\tilde \chi_{\varepsilon} f(t))\|_{L^\infty})\leq \varepsilon^\frac{1}{4}.
		\end{align*}
		where we denote $\tilde \chi_{\varepsilon}(x)=1-\chi(\varepsilon x)$.
	\end{lemma}
	\begin{proof}
		We have the equation
		\begin{align*}
			\partial_{t} f(t, x)=\int_{\mathbb{R}^d} \frac{\alpha \cdot \nabla f(t, x)-\delta_{\alpha} f(t, x)}{\left\langle\Delta_{\alpha} f(t, x)\right\rangle^{d+1}} \frac{d \alpha}{|\alpha|^{d+1}}:=N(f)(t,x).
		\end{align*}
		We have 
		\begin{align*}
			\partial_t (\tilde \chi_{\varepsilon_1}f)=N(\tilde \chi_{\varepsilon_1}f)+S(\tilde \chi_{\varepsilon_1},f),
		\end{align*}
		where $S(\tilde \chi_{\varepsilon_1},f)=\tilde \chi_{\varepsilon_1}N(f)-N(\tilde \chi_{\varepsilon_1}f)$ is the commutator. For any $i,j\in 1,\cdots, d$, we have 
		\begin{align*}
			\partial_t \partial_{ij}(\tilde \chi_{\varepsilon_1}f)=\partial_{ij}N(\tilde \chi_{\varepsilon_1}f)+\partial_{ij}S(\tilde \chi_{\varepsilon_1},f).
		\end{align*}
		Denote $\tilde \chi_{\varepsilon_1}f=g$, note that 
		\begin{align*}
			\partial_{ij}N(g)=&\int_{\mathbb{R}^d} \frac{\hat \alpha \cdot \nabla \partial_{ij}g}{\langle\Delta_\alpha g\rangle^{d+1}}\frac{d\alpha}{|\alpha|^d}-\int_{\mathbb{R}^d} \frac{\Delta_\alpha\partial_{ij}g}{\langle\Delta_\alpha g\rangle^{d+1}}\frac{d\alpha}{|\alpha|^d}\\
			&-(d+1)\int_{\mathbb{R}^d} \frac{(\hat \alpha \cdot \nabla \partial_{j}g-\Delta_\alpha \partial_j g)\Delta_\alpha g \Delta_\alpha \partial_ig}{\langle\Delta_\alpha g\rangle^{d+1}}\frac{d\alpha}{|\alpha|^d}\\
			&-(d+1)\int_{\mathbb{R}^d} \frac{(E_\alpha \partial_ig \Delta_\alpha g+E_\alpha g\Delta_\alpha \partial_i g)\Delta_\alpha \partial_j g}{\langle\Delta_\alpha g\rangle^{d+3}}\frac{d\alpha}{|\alpha|^d}\\
			&+(d+1)(d+3)\int_{\mathbb{R}^d} \frac{(E_\alpha g (\Delta_\alpha g)^2\Delta_\alpha \partial_i g\Delta_\alpha \partial_j g}{\langle\Delta_\alpha g\rangle^{d+5}}\frac{d\alpha}{|\alpha|^d}.
		\end{align*}
		Take $x_{i,j}$ such that $\partial_{ij}g(t,x_{ij})=\sup_x\partial_{ij}g(t,x)$. Denote $M_{ij}(t)=\sup_x\partial_{ij}g(t,x)$ and $A(t)=\|\nabla ^2g(t)\|_{L^\infty}$. Then one has $\nabla \partial_{ij} g(t,x_{ij})=0$ and $\Delta_\alpha \partial_{ij} g(t,x_{ij})\geq 0$. Moreover, one has 
		\begin{align*}
			&\left|\int_{\mathbb{R}^d} \frac{\hat \alpha \cdot \nabla \partial_{j}g\Delta_\alpha g \Delta_\alpha \partial_ig}{\langle\Delta_\alpha g\rangle^{d+1}}\frac{d\alpha}{|\alpha|^d}\right|\\
			&\leq RA(t)\int_{|\alpha|\geq 1} \frac{d\alpha}{|\alpha|^{d+1}}+\frac{1}{2}\int_{|\alpha|\leq 1} \hat \alpha \cdot \nabla \partial_{j}g\left(\frac{\Delta_\alpha g \Delta_\alpha \partial_ig}{\langle\Delta_\alpha g\rangle^{d+1}}-\frac{\Delta_{-\alpha} g \Delta_{-\alpha} \partial_ig}{\langle\Delta_{-\alpha} g\rangle^{d+1}}\right)\frac{d\alpha}{|\alpha|^d}.
		\end{align*}
		Observe that 
		\begin{align*}
			\left|\frac{\Delta_\alpha g \Delta_\alpha \partial_ig}{\langle\Delta_\alpha g\rangle^{d+1}}-\frac{\Delta_{-\alpha} g \Delta_{-\alpha} \partial_ig}{\langle\Delta_{-\alpha} g\rangle^{d+1}}\right|\leq  |\alpha|R^2.
		\end{align*}
		Hence 
		\begin{align*}
			&\left|\int_{\mathbb{R}^d} \frac{\hat \alpha \cdot \nabla \partial_{j}g\Delta_\alpha g \Delta_\alpha \partial_ig}{\langle\Delta_\alpha g\rangle^{d+1}}\frac{d\alpha}{|\alpha|^d}\right|\lesssim (1+R)^2A(t).
		\end{align*}
		Other terms can be estimated similarly, which leads to 
		\begin{align*}
			\left.\partial_{ij}N(g)\right|_{x=x_{ij}}\lesssim (1+R)^2A(t).
		\end{align*}
		On the other hand, we have $|\nabla \tilde \chi_{\varepsilon_1}|\leq \varepsilon_1$. It is easy to check that 
		\begin{align*}
			\|\partial_{ij}S(\tilde \chi_{\varepsilon_1},f)\|_{L^\infty}\lesssim \varepsilon_1^\frac{1}{2} (1+R)^3.
		\end{align*}
		Then we have for any $t\in[0,T]$
		\begin{align}\label{blabla}
			\frac{d}{dt} M_{ij}(t)\lesssim  (1+R)^2A(t)+\varepsilon_1^\frac{1}{2} (1+R)^3.
		\end{align}
		Take $\varepsilon_1$ small enough such that $\operatorname{supp}f_0\subset B_{\varepsilon_1^{-1}}(0)$, then $M_{ij}(0)=0$. 
		Integrate \eqref{blabla} in time we obtain
		\begin{align*}
			M_{ij}(t)\lesssim (1+R)^2\int_0^tA(\tau)d\tau+\varepsilon_1^\frac{1}{2} (1+R)^3t.
		\end{align*}
		We can estimate $m_{ij}(t)=\sup_x(-\partial_{ij}g(t,x))$ similarly. Summing up in $i,j$, we obtain 
		\begin{align*}
			A(t)\lesssim (1+R)^2\int_0^tA(\tau)d\tau+\varepsilon_1^\frac{1}{2} (1+R)^3t.
		\end{align*}
		By Gronwall's inequality we have 
		\begin{align*}
			A(t)&\leq \varepsilon_1^\frac{1}{2}(1+R)^3t+\varepsilon_1^\frac{1}{2}(1+R)^5\int _0^t \tau e^{\tau(1+R)^2}d\tau\\
			&\leq \varepsilon_1^\frac{1}{2}(1+R)^3T+\varepsilon_1^\frac{1}{2}(1+R)^3Te^{(1+R)^2T},\ \ \forall t\in[0,T].
		\end{align*}
		This implies the result in the lemma by taking $\epsilon_1$ small enough such that $(1+R)^3Te^{(1+R)^2T}\leq \epsilon_1^{-\frac14}$. 
	\end{proof}
	The following lemma bounds the diffusive term from above by strictly negative terms.
	\begin{lemma}\label{lemdissip}
		Under the assumption \eqref{con2},
		\begin{align*}
			&\int_{\mathbb{R}}\delta_\alpha f_x(y)-\delta_\alpha f_x(x)\frac{d\alpha}{\alpha^2}\\
			&\quad\quad\quad\leq -2\int_0^z\delta_\alpha\rho_t (z)+\delta_{-\alpha}\rho_t (z)\frac{d\alpha}{\alpha^2}+2\int_z^\infty\rho_t (\alpha+z)-\rho_t (\alpha)-\rho_t (z)\frac{d\alpha}{\alpha^2},
		\end{align*}
		where $z=|x-y|$.
	\end{lemma}
	\begin{proof}
		Note that 
		\begin{align*}
			f_x(\alpha)\leq f_x(x)+\rho_t(\alpha-y)-\rho_t(x-y),\  \forall \alpha\geq y,\\
			f_x(\alpha)\geq f_x(y)-\rho_t(x-\alpha)+\rho_t(x-y), \ \forall \alpha\leq x.
		\end{align*}
		By definition, for $\epsilon>0$ small enough, we have
		$$
		\int_{-\infty}^{-\epsilon}(\delta_\alpha f_x(y)-\delta_\alpha f_x(x))\frac{d\alpha}{\alpha^2}= -\int_{\epsilon}^{\infty}\frac{\rho_t(x-y)}{\alpha^2}d\alpha+\int_{\epsilon}^{\infty}\frac{f_x(x+\alpha)-f_x(y+\alpha)}{\alpha^2}d\alpha.
		$$
		We write the latter term as
		\begin{align*}
			&\int_{\epsilon}^{\infty}\frac{f_x(x+\alpha)-f_x(y+\alpha)}{\alpha^2}d\alpha=\int_{x+\epsilon}^{\infty}\frac{f_x(\alpha)}{(\alpha-x)^2}d\alpha-\int_{y+\epsilon}^{\infty}\frac{f_x(\alpha)}{(\alpha-y)^2}d\alpha\\
			&=\int_{x+\epsilon}^{\infty}f_x(\alpha)(\frac{1}{(\alpha-x)^2}-\frac{1}{(\alpha-y)^2})d\alpha-\int_{y+\epsilon}^{x+\epsilon}\frac{f_x(\alpha)}{(\alpha-y)^2}dz\\
			&\leq \int_{x+\epsilon}^{\infty}(f_x(x)+\rho_t(\alpha-y)-\rho_t(x-y))(\frac{1}{(\alpha-x)^2}-\frac{1}{( \alpha-y)^2})d\alpha-\int_{y+\epsilon}^{x+\epsilon}\frac{f_x(\alpha)}{(\alpha-y)^2}d\alpha.
		\end{align*}
		Direct calculation leads to
		\begin{align*}
			&\int_{x+\epsilon}^{\infty}(f_x(x)+\rho_t(\alpha-y)-\rho_t(x-y))(\frac{1}{(\alpha-x)^2}-\frac{1}{(\alpha-y)^2})d\alpha\\
			&=\int_{\epsilon}^{\infty}\frac{f_x(x)+\rho_t(x-y+\alpha)-\rho_t(x-y)}{\alpha^2}d\alpha-\int_{x-y+\epsilon}^{\infty}\frac{f_x(x)+\rho(\alpha)-\rho(x-y)}{\alpha^2}d\alpha\\
			&=\int_{\epsilon}^{x-y+\epsilon}\frac{f_x(x)+\rho_t(x-y+\alpha)-\rho_t(x-y)}{\alpha^2}d\alpha+\int_{x-y+\epsilon}^{\infty}\frac{\rho_t(x-y+\alpha)-\rho_t(\alpha)}{\alpha^2}d\alpha.
		\end{align*}
		Observe that 
		\begin{align*}
			&\int_{2x-y+\epsilon}^{\infty}\frac{f_x(x)+\rho_t(\alpha-y)-\rho_t(x-y)}{(\alpha-x)^2}d\alpha\\
			&\ \ \quad\quad -\int_{x+\epsilon}^{\infty}\frac{f_x(x)+\rho_t(\alpha-y)-\rho_t(x-y)}{(\alpha-y)^2}d\alpha-\int_{x-y+\epsilon}^{\infty}\frac{\rho_t(x-y)}{\alpha^2}d\alpha\\
			& =\int_{x-y+\epsilon}^{\infty}\frac{\rho_t(x-y+\alpha)-\rho_t(x-y)-\rho_t(\alpha)}{\alpha^2}d\alpha.
		\end{align*}
		Hence
		\begin{align*}
			&\int_{-\infty}^{-\epsilon}(\delta_\alpha f_x(y)-\delta_\alpha f_x(x))\frac{d\alpha}{\alpha^2}\\
			&\leq -\int_{\epsilon}^{\infty}\frac{\rho_t(x-y)}{\alpha^2}d\alpha+\int_{\epsilon}^{\infty}\frac{f_x(x)+\rho_t(x-y+\alpha)-\rho_t(x-y)}{\alpha^2}d\alpha\\
			&\quad\quad-\int_{x-y+\epsilon}^{\infty}\frac{f_x(x)+\rho_t(\alpha)-\rho_t(x-y)}{\alpha^2}d\alpha-\int_{\epsilon}^{x-y+\epsilon}\frac{f_x(y+\alpha)}{\alpha^2}d\alpha\\
			&\leq -\int_{\epsilon}^{x-y+\epsilon}\frac{\rho_t(x-y)}{\alpha^2}d\alpha+\int_{\epsilon}^{x-y+\epsilon}\frac{f_x(x)+\rho_t(x-y+\alpha)-\rho_t(x-y)}{\alpha^2}d\alpha\\
			&\quad\quad-\int_{\epsilon}^{x-y+\epsilon}\frac{f_x(\alpha+y)}{\alpha^2}d\alpha+\int_{x-y+\epsilon}^{\infty}\frac{\rho_t(x-y+\alpha)-\rho_t(x-y)-\rho_t(\alpha)}{\alpha^2}d\alpha.
		\end{align*}
		The integral in $(0,+\infty)$ is similar. Since $f_x(x)-f_x(y+\alpha)\leq \rho_t(x-y-\alpha)$, we get the result by taking $\epsilon\rightarrow 0$.  This completes the proof.
	\end{proof}
		{\bf{Acknowledgments.}}  Quoc-Hung Nguyen  is supported by the Academy of Mathematics and Systems Science, Chinese Academy of Sciences startup fund; CAS Project for Young Scientists in Basic Research, Grant No. YSBR-031;  and the National Natural Science Foundation of China (No. 12288201);  and the National Key R$\&$D Program of China under grant 2021YFA1000800. Ke Chen is supported by the Research Centre for Nonlinear Analysis at The Hong Kong Polytechnic University.
	
\end{document}